\ifx \labelsONmargin\undefined

  \ifx\RequirePackage\undefined
    \expandafter\ifx\csname amsart.sty\endcsname\relax
        \documentstyle[12pt,amssymb,amscd]{amsart}
    \fi

    \def\mathbb{\Bbb}
    \def\mathfrak{\frak}
    \def\mathbf{\bold}
    \ifx\boldsymbol\undefined	% I *think* some very old AMS* did not have it
      \def\boldsymbol#1{{\bold #1}}
    \fi

    \def\mathbit{\boldsymbol}

    \makeatletter
    \newenvironment{proof}{%
         \@ifnextchar[{%
                       \expandafter\let\expandafter\end@proof
                         \csname endpf*\endcsname
                         \my@proof
                      }{\let\end@proof\endpf\pf}%
        }{\end@proof}
    \def\my@proof[#1]{\@nameuse{pf*}{#1}}
    \makeatother

    \def\xrightarrow[#1]#2{@>{#2}>{#1}>}
    \def\xleftarrow[#1]#2{@<{#2}<{#1}<}

    \def\providecommand#1{\def#1}
    \def\emph#1{{\em #1}}
    \def\textbf#1{{\bf #1}}
  \else
      \ifx\usepackage\RequirePackage
      \else
        \documentclass[12pt]{amsart}
        \usepackage{amssymb,amscd}
      \fi

      \ifx\mathbit\undefined	% They can collect their senses and define it
        \DeclareMathAlphabet{\mathbit}{OML}{cmm}{b}{it}
      \fi

      \advance\oddsidemargin -0.1\textwidth
      \evensidemargin\oddsidemargin
      \def\Sb#1\endSb{_{\substack{#1}}}
      \def\Sp#1\endSp{^{\substack{#1}}}
  \fi
\makeatletter
\@ifundefined{DeclareFontShape} 
     {\@ifundefined{selectfont}
             {%    Here we are in non-NFSS
                \message{We need AmS-LaTeX, this version of LaTeX does not 
                        support it}\@@end
             }{%  Here we are in  NFSS1
                \def\mathcal{\cal}
                
                \def\pcyr{%
                        \def\default@family{UWCyr}%
                        \let\oldSl@\sl
                        \def\sl{\def\default@shape{it}\oldSl@}%
                        \cyracc
                        \language\Russian\family{UWCyr}\selectfont
                }
             }}{% Here we are in  NFSS2
                \DeclareFontEncoding{OT2}{}{\noaccents@}
                \DeclareFontSubstitution{OT2}{cmr}{m}{n}
                \DeclareFontFamily{OT2}{cmr}{\hyphenchar\font45 }
                \DeclareFontShape{OT2}{cmr}{m}{n}{%
                     <5><6><7><8><9><10>gen*wncyr %
                     <10.95><12><14.4><17.28><20.74><24.88> wncyr10 %
                }{}
                \DeclareFontShape{OT2}{cmr}{m}{it}{%
                     <5><6><7><8><9><10> gen * wncyi%
                     <10.95><12><14.4><17.28><20.74><24.88> wncyi10%
                }{}
                \DeclareFontShape{OT2}{cmr}{bx}{n}{%
                     <5><6><7><8><9><10> gen * wncyb%
                     <10.95><12><14.4><17.28><20.74><24.88> wncyb10%
                }{}
                \DeclareFontShape{OT2}{cmr}{m}{sl}{%
                     <-> ssub * cmr/m/it%
                }{}
                \DeclareFontShape{OT2}{cmr}{m}{sc}{%
                     <5><6><7><8><9><10>%
                     <10.95><12><14.4><17.28><20.74><24.88> wncysc10%
                }{}
                \DeclareFontFamily{OT2}{cmss}{\hyphenchar\font45 }
                \DeclareFontShape{OT2}{cmss}{m}{n}{%
                     <8><9><10> gen * wncyss%
                     <10.95><12><14.4><17.28><20.74><24.88> wncyss10%
                }{}
                
                \def\cyrencodingdefault{OT2}
                \def\pcyr{%
                        \cyracc
                        \let\encodingdefault\cyrencodingdefault
                        \language\Russian\fontencoding{OT2}\selectfont
                }
             }   

\@ifundefined{theorembodyfont} 
     {
        \def\theorembodyfont#1{\relax}
        \makeatletter
          \let\@@th@plain\th@plain
          \def\th@plain{ \@@th@plain \slshape }
        \makeatother
        \let\normalshape\relax
     }{}   

\ifx\cprime\undefined
     \def\cprime{$'$}
\fi
\makeatletter
  \def\@sect@my#1#2#3#4#5#6[#7]#8{%
\ifnum #2>\c@secnumdepth
   \let\@svsec\@empty
 \else
   \refstepcounter{#1}%
\edef\@svsec{\ifnum#2<\@m
             \@ifundefined{#1name}{}{\csname #1name\endcsname\ }\fi
\noexpand\rom{\csname the#1\endcsname.}\enspace}\fi
 \@tempskipa #5\relax
 \ifdim \@tempskipa>\z@ % then this is not a run-in section heading
   \begingroup #6\relax
   \@hangfrom{\hskip #3\relax\@svsec}{\interlinepenalty\@M #8\par}%
   \endgroup
   \if@article\else\csname #1mark\endcsname{%
        \ifnum \c@secnumdepth >#2\relax\csname the#1\endcsname. \fi#7}\fi
\ifnum#2>\@m \else
       \let\@tempf\\ \def\\{\protect\\}\addcontentsline{toc}{#1}%
{\ifnum #2>\c@secnumdepth \else
             \protect\numberline{%
               \ifnum#2<\@m
               \@ifundefined{#1name}{}{\csname #1name\endcsname\ }\fi
               \csname the#1\endcsname.}\fi
           #8}\let\\\@tempf
     \fi
 \else
  \def\@svsechd{#6\hskip #3\@svsec
    \@ifnotempty{#8}{\ignorespaces#8\unskip
       \ifnum\spacefactor<1001.\fi}%
        \ifnum#2>\@m \else
          \let\@tempf\\ \def\\{\protect\\}\addcontentsline{toc}{#1}%
            {\ifnum #2>\c@secnumdepth \else
              \protect\numberline{%
                \ifnum#2<\@m
                \@ifundefined{#1name}{}{\csname #1name\endcsname\ }\fi
                \csname the#1\endcsname.}\fi
             #8}\let\\\@tempf\fi}%
 \fi
\@xsect{#5}}
  \ifx\RequirePackage\undefined
  \let\@sect\@sect@my             % Cannot just comment the above
  \fi
  \def\th@remark@my{\theorempreskipamount6\p@\@plus6\p@
    \theorempostskipamount\theorempreskipamount
    \def\theorem@headerfont{\it}\normalshape}
  \ifx\RequirePackage\undefined
  \let\th@remark\th@remark@my
  \fi
\let\myLabel\@gobble
\def\labelsONmargin{\@mparswitchfalse\def\myLabel##1{\@bsphack\marginpar
                                  {\normalshape\tiny\rm Label ##1}\@esphack}}
\ifx \url\undefined
  \def\url#1{#1}%
\fi
\def\cyracc{\def\u##1{%\if \i##1\accent"24 i%
                \if \i##1\char"1A%
                \else \if I##1\char"12%
                \else \accent"24 ##1\fi\fi }%
\def\"##1{\if e##1{\char"1B}%
                \else \if E##1{\char"13}%
                \else \accent"7F ##1\fi\fi }%
\def\9##1{\if##1z\char"19 
\else\if##1Z\char"11 
\else\if##1E\char"03 
\else\if##1e\char"0B 
\else\if##1u\char"18 
\else\if##1U\char"10 
\else\if##1A\char"17 
\else\if##1a\char"1F 
\else\if##1p\char"7E 
\else\if##1P\char"5E 
\else\if##1Q\char"5F 
\else\if##1q\char"7F 
\else\if##1i\char"1A 
\else\if##1I\char"12 
\else\if##1N\char"7D 
\fi
\fi
\fi
\fi
\fi
\fi
\fi
\fi
\fi
\fi
\fi
\fi
\fi
\fi
\fi
}%
\def\cydot{{\kern0pt}}}%
\def\cydot{$\cdot$}
\ifx\Russian\undefined
        \def\Russian{0\relax
    \message{Don't know the hyphenation rules for Russian^^J
                        Please do INITeX with `input  russhyph' in the 
                        command line}%
                \gdef\Russian{0\relax}%
        }
\fi
\ifx\RequirePackage\undefined			% for 209
  \def\@putname#1#2#3#4{\def\@@ref{#3}\let\old@bf\bf
        \def\bf##1{\old@bf\if?\noexpand##1?{#4}\else##1\fi}%
	#1{#2}%
        \let\bf\old@bf}
\else						% for 2e
  \def\@putname#1#2#3#4{\def\@@ref{#3}%
	\let\old@reset@font\reset@font
	\def\reset@font##1##2{\old@reset@font##1\if?\noexpand##2?{#4}\else##2\fi}%
	#1{#2}%
        \let\reset@font\old@reset@font}
\fi
\let\my@ref=\ref
\def\ref#1{\@putname\my@ref{#1}{#1}{\tiny\rm\@@ref}}
\let\my@pageref=\pageref
\def\pageref#1{\@putname\my@pageref{#1}{#1}{\tiny\rm\@@ref}}
\let\my@cite=\cite
\def\cite#1{\@putname\my@cite{#1}{\@citeb}{\tiny\rm\@@ref}}
\makeatother
\fi
\relax 
\theoremstyle{plain} % for references in unnumbered theorems
\theorembodyfont{\sl}

\numberwithin{equation}{section}
\theorembodyfont{\rm}
\theoremstyle{definition}

\newtheorem{definition}{Definition}[section]

\theoremstyle{remark}

\newtheorem{remark}[definition]{Remark} %\renewcommand{\theremark}{}

\theoremstyle{plain} % for future references
\theorembodyfont{\sl}

\newtheorem{theorem}[definition]{Theorem}
\newtheorem{lemma}[definition]{Lemma}
\newtheorem{corollary}[definition]{Corollary}
\newtheorem{proposition}[definition]{Proposition}

\newcommand{\Hom}{\operatorname{Hom}}
\newcommand{\ext}{\operatorname{ext}}

\newcommand{\Ind}{\operatorname{Ind}}
\newcommand{\md}{\operatorname{mod}}

\newcommand{\Span}{\operatorname{Span}}

\newcommand{\End}{\operatorname{End}}
\newcommand{\soc}{\operatorname{soc}}
\newcommand{\q}{\mathfrak{q}}

\newcommand{\kk}{\mathfrak{k}}
\newcommand{\h}{\mathfrak{h}}
\newcommand{\bb}{\mathfrak{b}}
\newcommand{\g}{\mathfrak{g}}
\newcommand{\cc}{\mathfrak{c}}
\usepackage{youngtab}

\usepackage{graphics}
\usepackage{amssymb,amsmath,amsthm,amscd}
\usepackage{mathrsfs}
\usepackage{enumerate}
\usepackage{color}
\usepackage[all]{xy}
\usepackage{pictexwd}

\begin{document}
\bibliographystyle{amsplain}
\relax 

\title[Tensor representations of $\q(\infty)$]{Tensor representations of $\q(\infty)$}

%\authors{ Dimitar Grantcharov \thanks{Supported by NSA grant H98230-13-1-0245}  \address Department of Mathematics \\ University of Texas at Arlington \\ Arlington, TX 76019  \email grandim@uta.edu  \and Vera Serganova \thanks{Supported by  NSF grant 0901554} \address Department of Mathematics \\ University of California at Berkeley \\ Berkeley, CA 94720  \email serganov@math.berkeley.edu }
\author{Dimitar Grantcharov} 
\address{Department of Mathematics \\ University of Texas at Arlington \\ Arlington, TX 76019} \email{grandim@uta.edu}
\author{Vera Serganova} 
\address{Department of Mathematics \\ University of California at Berkeley \\ Berkeley, CA 94720} \email{serganov@math.berkeley.edu}
\begin{abstract} We introduce a symmetric monoidal category of modules over the direct limit queer superalgebra $\q (\infty)$. The category can be defined in two equivalent ways with the aid of the large annihilator condition. Tensor products of copies of the natural and the conatural representations are injective objects in  this category. We obtain the socle filtrations and formulas for the tensor products of the indecoposable injectives. In addition, it is proven that the category is Koszul self-dual. \\

\noindent 2000 MSC: 17B65, 17B10, 16G10.\\

\noindent Keywords and phrases: queer superalgebras, tensor representations, Koszul duality.
\end{abstract}

\maketitle
\section{Introduction}

Recently  new symmetric monoidal categories have attracted considerable attention. Among them are the categories $\operatorname{Trep} \mathfrak{g}$
of modules over direct limit $\mathfrak{g}$ of classical Lie algebras generated as abelian tensor categories by the natural and conatural representations.
Namely, $\mathfrak g$ is one of the following: $\mathfrak{gl}(\infty)=\displaystyle\lim_{\longrightarrow}\mathfrak{gl}(n)$, $\mathfrak{o}(\infty)=\displaystyle\lim_{\longrightarrow}\mathfrak{o}(n)$ and 
$\mathfrak{sp}(\infty)=\displaystyle\lim_{\longrightarrow}\mathfrak{sp}(n)$.  In \cite{DPS} it is proven
that these categories have enough injective objects and that every object has a finite injective resolution. Furthermore, the algebra of endomorphisms of an injecive cogenerator
is described explicitly. With the aid of this description, it follows that the categories are Koszul. Furthermore, it is shown  in \cite{SS} that these categories satisfy a natural universality property.

The categories $\operatorname{Trep} \mathfrak{g}$ of direct limits of basic classical Lie superalgebras $ \mathfrak{g} = \mathfrak{gl}(\infty|\infty)$ and $ \mathfrak{g} = \mathfrak{osp}(\infty|\infty)$ were studied in \cite{Serg}. 
It was shown there that no  new categories  appear, namely that the categories  $\operatorname{Trep}\mathfrak{gl}(\infty|\infty)$ and 
$\operatorname{Trep}\mathfrak{gl}(\infty)$ are equivalent and that the categories 
$\operatorname{Trep}\mathfrak{o}(\infty)$ and $\operatorname{Trep}\mathfrak{osp}(\infty|\infty)$ are equivalent as symmetric monoidal categories.
Furthermore, one can use the properties of the category $\operatorname{Trep}\mathfrak{osp}(\infty|\infty)$ to prove that $\operatorname{Trep}\mathfrak{o}(\infty)$ and
$\operatorname{Trep}\mathfrak{sp}(\infty)$ are equivalent as monoidal abelian categories. 

In contrast with $\mathfrak{gl}(\infty|\infty)$ and $\mathfrak{osp}(\infty|\infty)$, for the strange Lie superalgebras $\q(\infty)$
and $p(\infty)$ we obtain new 
interesting symmetric monoidal categories. We believe that these categories satisfy certain universality conditions analogous to the the category 
$\operatorname{Trep}\mathfrak{gl}(\infty)$ and $\operatorname{Trep}\mathfrak{o}(\infty)$. The case of $p(\infty)$ is discussed in \cite{Serg} and  \cite{Serg2}. 

The goal of this paper is to  investigate in detail the category 
$\operatorname{Trep} \mathfrak{q}(\infty)$ of the direct limit queer Lie superalgebra $\mathfrak{q}(\infty)$. We give two equivalent intrinsic definitions of $\operatorname{Trep} \mathfrak{q}(\infty)$ using the large annihilator condition. Then we classify the 
simple and indecomposable injective modules of  $\operatorname{Trep} \mathfrak{q}(\infty)$ and show that the category is Koszul self-dual. The latter is especially interesting since it is known that the category of 
finite-dimensional modules over $\q(n)$ is not Koszul, even more -  the algebra of endomorphisms of an injective cogenerator is not  quadratic, see \cite{MM}. In the present paper we also classify 
the blocks of  $\operatorname{Trep} \mathfrak{q}(\infty)$ and  express the Ext-groups between the simple objects using the shifted Littlewood--Richardson coefficients, \cite{St},
\cite{GJKKK}. 

Another motivation to study the category $\operatorname{Trep}\mathfrak{gl}(\infty)$ arises from the fact that the Lie superalgebras $\q(n)$ have very interesting representation theory and combinatorics. 
Representations of $\q(n)$ in the tensor algebra of the natural representation were originally studied by A. Sergeev, \cite{S}, \cite{S2}. 
He discovered a duality analogous to
the celebrated Schur-Weyl duality, often called the Sergeev duality. This duality relates the above representations with projective representations of the symmetric group, and the characters of these 
representations are given by Schur $Q$-functions, see \cite{Macd}. If one considers representations of $\q(n)$ in the tensor algebra of the natural representation and its dual,
the situation is more complicated. In particular, the representations are not completely reducible and the algebras of intertwining operators are not semisimple. This situation 
was studied in \cite{JK}, where the latter algebras are presented in a diagrammatic form. These algebras are  generalizations of Brauer and walled Brauer algebras.
The Koszul algebra which appears in our category, is a subalgebra of this diagrammatic algebra. This is related to the fact, that we have a tensor functor $\Gamma_n$
from our category $\operatorname{Trep} \mathfrak{q}(\infty)$ to the category of finite-dimensional $\q(n)$-modules but this functor does not map simple objects to simple objects.

We would like to remark that the category $\operatorname{Trep}\mathfrak{gl}(\infty|\infty)$ was used in \cite{EHS} as a technical tool for constructing the abelian envelope of the Deligne's category
$\operatorname{Rep}Gl(t)$ when $t$ is integer. It seems that a similar construction can be obtained for type $Q$ which we will address in a subsequent paper.

The organization of the paper is the following. In Section 2 we collect some useful results on associative superalgebras and finite-dimensional representations of $\q (n)$. The two equivalent definitions of  $\operatorname{Trep}\q (\infty)$ and a classification of its simple objects are included in Section 3. In Section 4 we classify the indecomposable injective objects of   $\operatorname{Trep}\q (\infty)$ and obtain their socle filtration. In this section we also prove that the category is a symmetric monoidal category. In Section 5 we compute the extension groups between the simple objects in $\operatorname{Trep}\q (\infty)$ and show that every object has a final injective resolution. We also derive a formula of the tensor product of the indecomposable injectives in terms of shifted Littlewood-Richardson coefficients. The Koszulity and self-dual Koszulity of the category is proven in Section 6.

\medskip

\noindent {\it Acknowledgements.} Both authors would like to express their gratitude for the excellent working conditions provided by  Mathematisches Forschungsinstitut Oberwolfach where most of this research was performed. The first author is partially supported by Simons Collaboration Grant 358245. The second author is partially supported by
the NSF Grant  DMS 1303301.

\section{Preliminaries}
In this paper we work in the categories of $\mathcal A$-modules for a Lie superalgebra or an associative superalgebra $\mathcal A$ over $\mathbb C$.
Thus, all objects are equipped with $\mathbb Z_2$-grading.
We use the notation $\Hom (\cdot, \cdot)$ for the supervector space of all $\mathcal A$-equivariant linear maps. For abelian categories we consider only morphisms that preserve parity, which we denote by $\hom(\cdot, \cdot)$. The Ext-groups  in the abelain category of $\mathcal A$-modules will be denoted by $\ext^i (\cdot, \cdot)$.

All multiplicities and dimensions will be considered as elements of $\mathbb Z[\varepsilon]/(\varepsilon^2-1)$. We set $\theta=1+\varepsilon$. Note that 
multiplication by $\theta$ is an injective map $\mathbb N[\varepsilon]/(\varepsilon^2-1)\to \mathbb N[\varepsilon]/(\varepsilon^2-1)$. Hence we say that 
$\zeta\in \mathbb N[\varepsilon]/(\varepsilon^2-1)$ is divisible by $\theta$ if $\zeta=\xi\theta$ for some (unique) $\xi\in\mathbb N[\varepsilon]/(\varepsilon^2-1)$ 
and we set $\xi=\frac{\zeta}{\theta}$.  
At the level of Grothendieck rings we let
$\varepsilon[M]=[\Pi M]$, where $\Pi$ is the switch of parity functor.

We next state the super-analogue of the classical Schur's Lemma. For the proof, see  \S1.1.6 in \cite{K}. 
\begin{lemma}\label{lem:Schur} Let $\mathcal A$ be a finite or countable-dimensional superalgebra over $\mathbb C$ and $M$ be a simple $\mathcal A$-module.
Then either $\End(M)=\mathbb C$ or  $\End(M)$ is isomorphic to the superalgebra $\mathbb C[\xi]/(\xi^2-1)$ with an odd generator $\xi$.
\end{lemma}

We say that a simple $\mathcal A$-module is of M{\it -type} if  $\End(M)=\mathbb C$ or, equivalently, if $M$ and $\Pi M$ are not isomorphic. Alternatively, a simple $\mathcal A$-module is of Q-{\it type} if  $\End(M)=\mathbb C[\xi]/(\xi^2-1)$ or, equivalently, if  $M$ and $\Pi M$ are isomorphic. From now on we set $C_1=\mathbb C[\xi]/(\xi^2-1)$.

Let $\mathcal A$ and $\mathcal B$ be two superalgebras, $M$ be a simple $\mathcal A$-module and $N$ be a simple $\mathcal B$-module.
If both $M$ and $N$ are of Q-type, we set
$$M\widehat{\boxtimes}N:=M\otimes_{C_1}N.$$
Then $M\widehat{\boxtimes}N$ is a simple $\mathcal A\otimes\mathcal B$-module. We have the natural decomposition
$$M{\boxtimes}N\simeq M\widehat{\boxtimes}N\oplus \Pi(M\widehat{\boxtimes}N),$$
and the embedding 
$C_1\hookrightarrow \End_{\mathcal A\otimes \mathcal B}(M\widehat{\boxtimes}N)$ defined by $\xi\mapsto \xi\otimes 1$.

Let $\mathcal A=U(\mathfrak k)$ be the universal enveloping of a superalgebra $\kk$, then $\mathcal A$ is a Hopf superalgebra
and $M\otimes N$ is equiped with an $\mathcal A$-module structure. 
If $M$ and $N$ are of Q-type, then we define
$$M\widehat \otimes N:=M\otimes_{C_1}N.$$

We will also need the following general lemma.

\begin{lemma} \label{lem:gensuper}
Let ${\mathcal A}$ be a semisimple associative unital superalgebra over ${\mathbb C}$ and let $e \in {\mathcal A}$ be a primitive idempotent of ${\mathcal A}$.
\begin{itemize}
\item[(i)] The following identity holds
\begin{equation} \label{equ:sumsimples}
{\mathcal A} = \bigoplus_{L \in {\rm Irr}\, {\mathcal A}} L \boxtimes_{\End (L)} L^*, 
\end{equation}
where ${\rm Irr}\, {\mathcal A}$ denote the set of isomorphism classes of irreducible left ${\mathcal A}$-modules.
\item[(ii)] ${\mathcal A} e$ is an irreducible ${\mathcal A}$-module.
\item[(iii)]  Let $M$ be a finite-dimensional ${\mathcal A}$-module. Then
$$
[M: {\mathcal A} e] = \frac{\dim eM }{\dim \End_{\mathcal A} ({\mathcal A} e)}
$$
and
$$\dim eM=\dim\Hom_{\mathcal A}({\mathcal A} e,M).$$
\end{itemize}
\end{lemma}

In what follows we also use several facts about representation theory of the Lie superalgebra $\q(n)$.
We call a weight $\kappa$ integral dominant if the irreducible $\q(n)$-module $L_n(\kappa)$ with highest weight
$\kappa$ is finite-dimensiional and can be lifted to the representation of the algebraic supergroup $Q(n)$.
It follows from \cite{P} that the integral dominant weights are of the form $a_1\delta_1+\dots+a_n\delta_n$, with $a_i\in\mathbb Z$
satisfying the conditions
\begin{enumerate}
\item if $a_i\neq 0$, then $a_i>a_{i+1}$;
\item if $a_i=0$, then $a_i\geq a_{i+1}$.
\end{enumerate}

Let $M_n(\kappa)$ denote the Verma module with highest weight $\kappa$ and $X_n(\kappa)$ be the maximal finite-dimensional quotient of $M_n(\kappa)$.
Then $X_n(\kappa)$ has the following geometric interpretation. Let $P_\kappa$  be the maximal parabolic subgroup of $Q(n)$ such that $\kappa$ induces a one-dimensional
representation of the even subgroup $(P_\kappa)_0$. Let $\mathcal O(\kappa)$ be the vector bundle over $Q(n)/P_\kappa$ corresponding to  the irreducible representation of $P_\kappa$ with character $-\kappa$. Then
$$X_n(\kappa)\simeq H^0(Q(n)/P_\kappa, \mathcal O(\kappa))^*$$ (see for example Lemma 2 in \cite{GS}).
Certain bounds for the multiplicities of the simple $Q(n)$-subquotients  of $ H^i(Q(n)/P_\kappa, \mathcal O(\kappa))^*$ can be deduced from \cite{PS}. 
We will use the following statement about the structure of $X_n(\kappa)$ which follows from these bounds.

\begin{proposition}\label{prop:cohomology} Let $\kappa=a_1\delta_1+\dots+a_n\delta_n$ be an integral dominant weight such that
$a_1>a_2>\dots>a_k>0$, $a_{k+1}=\dots=a_{k+r}=0$, $0>a_{k+r+1}>\dots>a_n$.
\begin{enumerate}
\item[(i)] The length of $X_n(\kappa)$ is at most $2^{a_1+\dots+a_k-a_{k+r+1}-\dots-a_n}$;
\item[(ii)] Assume that $r>a_1+\dots+a_k-a_{k+r+1}-\dots-a_n$ and $[X_n(\kappa):L_n(0)]\neq 0$. Then $\kappa=0$ or $\kappa = \delta_1-\delta_n$.
\end{enumerate}
\end{proposition}

\section{ Category $\operatorname{Trep}\q(\infty)$} 
\subsection{Lie superalgebra $\q(\infty)$.}
Let $V=V_0\oplus V_1$ and $W=W_0\oplus W_1$ be two
countable-dimensional supervector spaces, equipped with an even non-degenrate pairing $$(\cdot,\cdot):W\times V\to \mathbb C.$$
Denote by $1_W$ and $1_V$ the identity endomorphisms on $W$ and $V$, respectively. Let $P:V\to V$ be an odd linear operator such that $P^2=-1_V$. Define the action of $P$ on $W$ by setting
$$(Pw,v)=-(-1)^{p(w)}(w,Pv).$$
Note that $P^2|_{W}=1_W$.

Following \cite{Mc}, we fix dual bases $\{e_i,\,i\in\mathbb Z\setminus 0\}$ of $V_0$ and $\{f_i,\,i\in\mathbb Z\setminus 0\}$ of $W_0$ such that
$(f_i,e_j)=\delta_{ij}$. Set $\bar e_i=Pe_i$ and $\bar f_i=P f_i$. Then we have $(\bar f_i,\bar e_j)=\delta_{ij}$.

Let $\q(\infty)$ be the Lie superalgebra of finitary linear operators in $\End(V)\oplus\End(W)$ which  satisfy
$$(Xw,v)=-(-1)^{p(w)p(X)}(w,Xv),\quad [X,P]=0.$$
Henceforth we set  $\g = q(\infty)$.

One can easily see that $V$ and $W$ are $\g$-modules. We denote by $T^{p,q}$ the tensor product $V^{\otimes p}\otimes W^{\otimes q}$ which is also a $\g$-module.
One can easily check that
$$T^{1,1}=V\otimes W\simeq\g\oplus\Pi\g,$$ where $\g$ is considered as the adjoint $\g$-module.
We also have that
$$
\g = \Span_{\mathbb C} \{e_i\otimes f_j+\bar e_i\otimes \bar f_j,  e_i\otimes \bar f_j+\bar e_i\otimes f_j \; | \; i,j\in\mathbb Z\setminus 0\}.
$$

Let $\g_n\simeq\q(n)$ be the Lie subalgebra spaned of $e_i\otimes f_j+\bar e_i\otimes \bar f_j$
and $e_i\otimes \bar f_j+\bar e_i\otimes f_j$ for all $-\lfloor\frac{n}{2}\rfloor \leq i,j\leq \lceil\frac{n}{2}\rceil$. 
Then $\q(\infty)$ is the direct limit 
$$\q(\infty)=\lim_{\longrightarrow}\g_n.$$ 
Denote by $\cc_n$ the centralizer of $\g_n$ in $\g$. Note that for all $n$,  $\cc_n$ is isomorphic to $\q(\infty)$.

\subsection{Large annihilator condition}
Define a left exact functor $\Gamma_n:\g-\md\to \g_n-\md$ by setting $$\Gamma_n(M):=M^{\cc_n}.$$
The direct limit $$\Gamma:=\displaystyle\lim_{\longrightarrow}\Gamma_n:\g-\md\to\g-\md$$
is also a left exact functor.

Clearly, we have a canonical embedding $\Gamma(M)\hookrightarrow M$. We say that $M$ satisfies the {\it large annihilator condition} if $\Gamma(M)=M$.
Note that modules satisfying this condition form an abelain subcategory in $\g-\md$. 
Furthermore, one can easily see that if $M$ and $N$ satisfy the large annihilator condition, the tensor product $M\otimes N$ also satisfies it.
In particular, $V$, $W$, and hence $T^{p,q}$, satisfy the large annihilator condition. The following lemma is straightforward.

\begin{lemma}\label{lm:gammaadjoint} Let $M$ and $Y$ be $\g$-modules. Assume that $M$ satisfies the large annihilator condition. Then there is a canonical isomorphism
$$\Hom_\g(M,Y)\simeq\Hom_{\g}(M,\Gamma(Y)).$$
\end{lemma}

We call a $\g$-module $M$ {\it integrable} if for any $n>0$ it can be lifted to a representation of algebraic group $Q(n)$ with the Lie superalgebra $\g_n$.

\begin{definition} The category $\operatorname{Trep}\g$  of \emph{tensor representations} of $\g$ is the full subcategory of $\g-\md$ whose objects $M$ satisfy the following properties.
\begin{enumerate}
\item $M$ is an integrable $\g$-module.
\item $M$ has finite length.
\item $M$ satisfies the large annihilator condition.
\end{enumerate}
\end{definition}

It is clear that $T^{p,q}$ satisfies (1) and (3). Furthermore, the restriction of $T^{p,q}$ to $\g_0$ has finite length, see Theorem 2.3 in \cite{PSt}. Hence  
$T^{p,q}$ has finite length as a $\g$-module. Therefore $T^{p,q}$ is an object of $\operatorname{Trep}\g$.

Consider the Cartan subalgebra $\h$ of $\g$ spanned by  $e_i\otimes f_i+\bar e_i\otimes \bar f_i$
and $e_i\otimes \bar f_i+\bar e_i\otimes f_i$ for $i\in\mathbb Z\setminus 0$. Note that the even part $\h_0$ of $\h$ is the diagonal subalgebra of $\g$.  Let  $\{\varepsilon_i, i\in\mathbb Z\setminus 0\}$ be the system in $\h^*_0$ dual to the basis $e_i\otimes f_i+\bar e_i\otimes \bar f_i$ of $\h_0$. 
Denote by $\Lambda$ the $\mathbb Z$-linear span of $\{\varepsilon_i, i\in\mathbb Z\setminus 0\}$.

\begin{lemma}\label{lm:weights} If $M\in \operatorname{Trep}\g$, then $M$ is $\h_0$-semisimple and the weights of $M$ belong to $\Lambda$.
\end{lemma}
\begin{proof} Note that $M$ is semisimple over the Cartan subalgebra $\h_n$ of $\g_n$. Together with the large annihilator condition this implies
that $M$ is $\h$-semisimple since $\h$ is the direct limit of $\h_n$. 
\end{proof}

\subsection{Highest weight category}
Throughout the paper we will use the following ``exotic'' total order on $\mathbb Z\setminus 0$:
$$1\prec2\prec\dots\prec-2\prec-1.$$ In particular, the positive numbers are smaller than the negative ones.

Let $\mathfrak n\subset \g$ be the subalgebra spanned by  $e_i\otimes f_j+\bar e_i\otimes \bar f_j$
and $e_i\otimes \bar f_j+\bar e_i\otimes f_j$ for all $i\prec j$. Then $\bb=\mathfrak n\oplus\h$ is a Borel subalgebra of $\g$ and we can  define the category $\mathcal O$ with respect to $\bb$. More precisely,
$\mathcal O$ is the full subcategory of $\g$-modules consisting of finitely generated modules that are semisimple over $\h_0$, and that are $\mathfrak n$-locally nilpotent.

We denote by $\mathcal O'$ the full subcategory of $\mathcal O$ consisting of  modules which satisfy the 
large annihilator condition, and by $\mathcal O'_{int}$ the subcategory of $\mathcal O'$  of integrable modules.
It is easy to check that Lemma \ref{lm:weights} holds for the category $\mathcal O'$ , namely, that that the weights of all modules in  $\mathcal O'_{int}$ belong to $\Lambda$.

Suppose $L$ is a simple highest weight module in $\mathcal O'_{int}$. Then by \cite{P}
the highest weight of $L$ is of the form 
\begin{equation}\label{eqn:form}
a_1\varepsilon_1+\dots+a_k\varepsilon_k-a_{-l}\varepsilon_{-l}-\dots- a_{-1}\varepsilon_{-1}
\end{equation}
with positive integers $a_i$ such that $a_1>\dots >a_k$ and $a_{-1}>\dots >a_{-l}$.
Hence we have a bijection between the dominant weights in $\Lambda$ and the strict bipartions $(\lambda,\mu)$, where $\lambda=(a_1,\dots,a_k)$
and $\mu=(a_{-1},\dots,a_{-l})$.

For any strict partition $\lambda$ of $r$ we set $|\lambda|=r$, denote by $l(\lambda)$ the number of parts  (nonzero components) of $\lambda$, and by  $p(\lambda)$ the parity of $l(\lambda)$. For a strict bipartition $(\lambda,\mu)$ we set
$$p(\lambda,\mu)=p(\lambda)+p(\mu).$$

For simplicity, for small (bi)partitions, we will use their corresponding Young tableau. For example $\square$ will denote the strict partition $(1)$, and $(\square, \square)$ will stand for the strict bipartition $((1), (1))$.

\begin{lemma}\label{lem:parity} If $p(\lambda,\mu)=0$, then there exist two up to isomorphism simple modules $V(\lambda,\mu)$ and $\Pi V(\lambda,\mu)$ in $\mathcal O'_{int}$
with highest weight $(\lambda,\mu)$.
If $p(\lambda,\mu)=1$ then there is a unique up to isomorphism simple module $V(\lambda,\mu)$  in $\mathcal O'_{int}$ with highest weight $(\lambda,\mu)$, and this module is of {\rm Q}-type.   
\end{lemma}
\begin{proof} Let $k = l(\lambda)$ and $l = l(\mu)$. Let $M$  be a simple module of highest weight $(\lambda,\mu)$, and let  $C(\lambda,\mu)$ be the $(\lambda,\mu)$-weight space of $M$. Then $C(\lambda,\mu)$ is a simple $U(\h)$-module.
It is easy to see that $e_i\otimes f_i+\bar e_i\otimes \bar f_i$ and $e_i\otimes \bar f_i+\bar e_i\otimes f_i$ act by zero on  $C(\lambda,\mu)$
if $k\prec i\prec {-l}$. Thus $C(\lambda,\mu)$ is a simple module over the Clifford algebra with $k+l$ generators. The statement follows from the theory of Clifford algebras. Namely, if $k+l$ is even, then the corresponding 
Clifford algebra is a matrix algebra equipped with $\mathbb Z_2$-grading and hence it has two up to isomorphism simple modules, $V$ and $\Pi V$. If $k+l$ is odd,
the Clifford algebra is a direct sum of two matrix algebras, however it is simple as a superalgebra and has unique up to isomorphism simple module $V \simeq \Pi V$.
\end{proof}

\begin{lemma}\label{lm:finitelength} Every module in $\mathcal O'_{int}$ has  finite length.
\end{lemma}
\begin{proof} For a bipartition $(\lambda, \mu)$, denote by $M(\lambda,\mu)$ the corresponding Verma module. Let $X(\lambda,\mu)$ be the maximal quotient of $M(\lambda,\mu)$  which is in $\mathcal O'_{int}$. 
Since every module in $\mathcal O'_{int}$ has a finite filtration by highest weight modules, it suffices to check that
$X(\lambda,\mu)$ has finite length. 

Let $n>|\lambda|+|\mu|$. Let $Y_n(\lambda,\mu)$ be the $\g_n$-submodule of $X(\lambda,\mu)$ generated by a highest weight vector of $X(\lambda,\mu)$.
Then $$X(\lambda,\mu)=\lim_{\longrightarrow}Y_n(\lambda,\mu).$$
On the other hand, $Y_n(\lambda,\mu)$ is a quotient of $X_n(\lambda,\mu)$.
By Proposition \ref{prop:cohomology}(i) the length of $X_n(\lambda,\mu)$ stabilizes. Hence  $X(\lambda,\mu)$ has finite length. 
\end{proof}

\subsection{Polynomial representations and Sergeev duality} By definition, the \emph{polynomial representations} of $\g$ are  those which occur in tensor powers of $V$.
We recall some facts related to the Sergeev duality. It is proven in \cite{S} that the centralizer ${\mathcal{H}}_r$ of $\g$ in $V^{\otimes r}$ is a semisimple
superalgebra which we call the \emph{Sergeev algebra}. 
Irreducible representations of $\mathcal H_r$ (up to change of parity) are parametrized by strict partitions of size $r$. We denote by
$S(\lambda)$ the irreducible representation of the Sergeev algebra ${\mathcal{H}}_r$ associated with $\lambda$.
Note that $S(\lambda)$ is of M-type (respectively, Q-type) if $p(\lambda)=0$ (respectively, $p(\lambda)=1$). By $e(\lambda)$ we denote a primitive idempotent
of $\mathcal H_r$ such that $\mathcal H_re(\lambda)\simeq S(\lambda)$.

For any $r>0$ we have a decomposition:\footnote{Although the result of Sergeev is for finite rank queer Lie superalgebras, it is easy to extend it to $\q (\infty)$ by taking direct limits.}
\begin{equation}\label{eq:sergeev}
V^{\otimes r}=\bigoplus_{p(\lambda)=0} V(\lambda, \emptyset)\boxtimes S(\lambda)\oplus\bigoplus_{p(\lambda)=1} V(\lambda, \emptyset)\widehat{\boxtimes}S(\lambda),
\end{equation}
where $\lambda$ runs over the set of all strict partition of $r$.

Similarly, we have
\begin{equation}\label{eq:sergeev1}
W^{\otimes r}=\bigoplus_{p(\lambda)=0} V(\emptyset, \lambda)\boxtimes S(\lambda)\oplus\bigoplus_{p(\lambda)=1} V(\emptyset, \lambda)\widehat{\boxtimes}S(\lambda).
\end{equation}

For simplicity, we set $V(\lambda):=V(\lambda,\emptyset)$ and $W(\lambda):= V(\emptyset,\lambda)$.

\subsection{Littlewood-Richardson coefficients}
By $f_{\lambda, \nu}^{\mu}$ we denote the Littlewood-Richardson coefficients  of type $Q$:
$$f_{\lambda, \nu}^{\mu} =\dim\Hom_{\mathfrak g} (V(\mu),V(\lambda) \otimes V(\nu)).$$

Another way to define Littlewood-Richardson coefficients is by using the branching law for the Sergeev algebra.
Henceforth we set $\mathcal H_{p,q} = \mathcal H_p \otimes \mathcal H_q$.

\begin{lemma}\label{lem:LRind} If $|\lambda| = p$ and $|\nu| = r$, then
$$ f_{\lambda, \nu}^{\mu}=
\dim\Hom_{\mathcal H_{p,r}}(S(\lambda)\boxtimes S(\nu), S(\mu))= \dim\Hom_{\mathcal H_{p+r}}(\Ind^{\mathcal H_{p+r}}_{\mathcal H_{p,r}}S(\lambda)\boxtimes S(\nu)).
$$
\end{lemma}
\begin{proof}
Then we have
$$V(\lambda)=e(\lambda)V^{\otimes p},\, V(\nu)=e(\nu)V^{\otimes r}$$
and
$$V(\lambda)\otimes V(\nu)=e(\lambda)\otimes e(\nu) (V^{\otimes (p+r)}).$$ 
By Sergeev duality we obtain
$$V(\lambda)\otimes V(\nu)=
\bigoplus_{|\mu|=p+r,p(\mu)=0} V(\mu)\boxtimes (e(\lambda)\otimes e(\nu)(S(\mu)))
\oplus\bigoplus_{|\mu|=p+r,p(\mu)=1} V(\mu)\widehat\boxtimes (e(\lambda)\otimes e(\nu)(S(\mu))).$$
If  $p(\lambda)p(\nu)=0$, then $e(\lambda)\otimes e(\nu)$ is a primitive idempotent in $\mathcal H_{p,r}$. 
By Lemma \ref{lem:gensuper}(iii) we have that
$$\dim e(\lambda)\otimes e(\nu)(S(\mu))=\dim\Hom_{\mathcal H_{p,r}}(S(\lambda)\boxtimes S(\nu), S(\mu)).$$
If $p(\lambda)p(\nu)=1$, then $e(\lambda)\otimes e(\nu)$ is a sum of two primitive idempotents corresponding to two irreducible representations of 
$\mathcal H_{p,r}$ such that one is obtained from the other by parity switch.
We again have that
$$\dim e(\lambda)\otimes e(\nu)(S(\mu))=\dim\Hom_{\mathcal H_{p,r}}(S(\lambda)\boxtimes S(\nu), S(\mu)).$$

The second equality is a consequence of Frobenius reciprocity.
\end{proof}
\begin{corollary}\label{cor:LR} If $f_{\lambda, \nu}^{\mu}\neq 0$, then $|\lambda|+|\nu|=|\mu|$.
\end{corollary}

Note that by Theorem 1.11 in \cite{GJKKK}
\begin{equation}    \label{equ:tensorproduct}
f_{\square, \nu}^{\mu}  = \begin{cases} 0, \,\,\text{if}\,\, \nu \notin \mu - \square,\\ \theta^{p(\nu)p(\mu)} \theta
,\,\,\text{if}\,\,  \nu \in \mu - \square. \end{cases}
\end{equation}

\subsection{The category $\mathcal O'_{int}$ coincides with $\operatorname{Trep}\g$}

\begin{lemma}\label{lem:finitegen} Let $M$ be a $\g$-module isomorphic to $V^{\oplus n}$. 
Then there exists a subspace $U\subset V$ with $\dim U = n\theta$, such that the symmetric algebra $S(V^{\oplus n})$ is generated
by $S(U^{\oplus n})$ as $\g$-module.
\end{lemma}
\begin{proof}
We use the isomorphism of $\g$-modules $$S(M)\simeq \bigoplus_{r_1,\dots,r_n\in{\mathbb Z}_{\geq 0}} S^{r_1}(V)\otimes\dots\otimes S^{r_n}(V).$$
Furthermore, if $V(\lambda)$ occurs in $S^{r_1}(V)\otimes\dots\otimes S^{r_n}(V)$, then $\lambda$ has at most $n$ rows. To show this we use the fact that all $V(\mu)$ that appear as direct summands of $V(\eta) \otimes S^r(V)$ have the property that $\mu - \eta$ is contained in a horizontal $r$-strip. For the latter we use the Pieri formula for Schur $P$-functions (see for example (5.7) in \S III.5 of \cite{Macd}) and the fact that the character of  $V(\lambda)$ is a multiple of the corresponding Schur $P$-function (and also of the $Q$-function).

Therefore the highest weight vectors
belong to $S^{r_1}(U)\otimes\dots\otimes S^{r_n}(U)$, where $U$ is the span of $e_i$ and $\bar e_i$ for $i=1,\dots, n$.
\end{proof}

\begin{remark}\label{rm:changeborel} Let $G$ be the group of all linear operators on $V\oplus W$ that preserve the pairing $(\cdot, \cdot)$, and that commute  with $P$. 
Then $G$ is a subgroup of the group of 
automorphisms of $\g$. Like in the case of $\mathfrak{gl} (\infty)$ (see Theorem 3.4 in \cite{DPS}) , the large annihilator condition implies that for any $\gamma\in G$, the twisted module $M^\gamma$ is 
isomorphic to $M$. Let ${\bf W}$ denote the normalizer of $\h$ in $G$. Then for any $s\in {\bf W}$, if $M$ is a highest weight module with respect to $s(\bb)$ it is a 
highest weight module with respect to $\bb$. 
\end{remark}

\begin{lemma}\label{lm:simplemodule} Every simple module in $\operatorname{Trep}\g$ is isomorphic to $V(\lambda,\mu)$ or $\Pi V(\lambda,\mu)$.
\end{lemma}
\begin{proof} Let $L$ be a simple module in $\operatorname{Trep}\g$. It suffices to prove the existence of a $\bb$-singular vector in $L$.

Let $v\in L$ be a non-zero weight vector. It is annihilated by some $\cc_n$. We consider the parabolic subalgebra $\mathfrak p$ of $\g$ with Levi
part $\mathfrak l=\g_n\oplus\cc_n$ and whose abelian nilradical $\mathfrak m$ is isomorpic to
$W_n\widehat\boxtimes V'$, where $V'$ is the standard $\cc_n$-module and $W_n$ is the costandard $\g_n$-module. In particular,
$\mathfrak m$ is isomorphic to $(V')^{\oplus n}$ as a $\cc_n$-module.
By Lemma \ref{lem:finitegen}, there exists a finite dimensional subspace $\mathfrak m'\subset \mathfrak m$ such that
$U(\mathfrak m)=S(\mathfrak m)$ is generated over $\cc_n$ by $S(\mathfrak m')$.
Since $L$ is integrable, the abelian subalgebra $\mathfrak m'$ acts locally nilpotently and therefore for some $N\geq 0$ we have
$S^N(\mathfrak m') v=0$. But then $S^N(\mathfrak m) v=0$.  
The latter implies that $L^{\mathfrak m}\neq 0$. 

Since $L$ is irreducible, $L^{\mathfrak m}$ is an irreducible $\mathfrak l$-module. On the other hand 
we note that $L^{\mathfrak m}$ is isomorphic to a $\cc_n$-submodule of $S^k(\mathfrak m)$ for some $k$. 
Hence $L^{\mathfrak m}$ contains a $\bb\cap\mathfrak l$-singular vector. 

Now we pick a $\bb\cap\mathfrak l$-singular vector $w$ in $L^{\mathfrak m}\neq 0$. Let $\bb'=(\bb\cap\mathfrak l) \oplus \mathfrak m$. 
Then $w$ is a $\bb'$-singular vector, and hence
$L$ is a highest weight module with respect to the borel subalgebra $\bb'$.

It is not difficult to see that $\bb'=s(\bb)$ for some $s\in {\bf W}$.
Thus the statement follows from Remark \ref{rm:changeborel}.
\end{proof}

\begin{theorem}\label{th:twocategories} The category $\operatorname{Trep}\g$ coincides with the category $\mathcal O'_{int}$.
\end{theorem}
\begin{proof}
Lemma \ref{lm:finitelength} implies that $\mathcal O'_{int}$ is a subcategory of $\operatorname{Trep}\g$.
The inclusion $\operatorname{Trep}\g\subset \mathcal O'_{int}$ follows from Lemma \ref{lm:simplemodule}.
\end{proof}

\section{ Injective modules in $\operatorname{Trep}\g$}

\subsection{Decomposition of mixed tensor powers}
For any strict bipartition $(\lambda,\mu)$ we define the $\g$-module 
$$Z(\lambda,\mu):=\begin{cases}V(\lambda)\otimes W(\mu)\,\,\text{if}\,\,p(\lambda)p(\mu)=0\\
V(\lambda)\widehat{\otimes} W(\mu)\,\,\text{if}\,\,p(\lambda)p(\mu)=1\end{cases}.$$ 
It clear from Sergeev duality that $Z(\lambda,\mu)$ is a submodule of $T^{p,q}$ for $p=|\lambda|$, $q=|\mu|$. 

Let $S(\lambda,\mu)$ be the $\mathcal H_{|\lambda|, |\mu|}$-module defined by
$$S(\lambda,\mu):=\begin{cases}S(\lambda)\boxtimes S(\mu)\,\,\text{if}\,\,p(\lambda)p(\mu)=0\\
S(\lambda) \widehat{\boxtimes} S(\mu)\,\,\text{if}\,\,p(\lambda)p(\mu)=1\end{cases}.$$

Sergeev's duality (\ref{eq:sergeev}) implies the following decomposition
\begin{equation}\label{equ:decomp}
T^{p,q}=\bigoplus_{|\lambda|=p,|\mu|=q,p(\lambda,\mu)=0} Z(\lambda,\mu)\boxtimes S(\lambda,\mu)\oplus 
\bigoplus_{|\lambda|=p,|\mu|=q,p(\lambda,\mu)=1} Z(\lambda,\mu)\widehat{\boxtimes}S(\lambda,\mu).
\end{equation}

Moreover, we have the following identities involving the primitive idempotents.
\begin{equation} \label{equ:zidempotent}
e(\lambda)\otimes e(\mu) \left( T^{p,q}  \right) =\begin{cases}Z(\lambda,\mu),\,\text{if}\,\,p(\lambda)p(\mu)=0\\ 
Z(\lambda,\mu)\oplus \Pi Z(\lambda,\mu),\,\text{if}\,\,p(\lambda)p(\mu)=1\end{cases}.
\end{equation}

\subsection{General properties of the functor $\Gamma_n$}
We now prove a lemma that is somewhat surprising.
\begin{lemma}\label{lm:invariants} Let $M$ be a $\g$-module satisfying the large annihilator condition. Then $\Hom_{\g}(\mathbb C,M)=\Hom_{\g_0}(\mathbb C,M)$. 
\end{lemma}
\begin{proof} We have the obvious inclusion $\Hom_{\g}(\mathbb C,M)\subset\Hom_{\g_0}(\mathbb C,M)$. To prove that equality we show that $M^{\g_0} \subset M^{\g}$.
Let $v\in M^{\g_0}$. By the large annihilator condition $(\cc_n)_1v=0$ for some $n$. But $(\cc_n)_1$ and $\g_0$ generate $\g$. Hence $\g v=0$.
\end{proof}

\begin{corollary}\label{cor:aboutgamma} For any $n>0$ and any module $M$ in  $\operatorname{Trep}\g$ we have
$\Gamma_n(M)=M^{(\cc_n)_0}$.
\end{corollary}
\begin{proof} The statement follows by restricting $M$ to $\cc_n$ and using Lemma \ref{lm:invariants}.
\end{proof}

Consider the restriction functor $\operatorname{Trep}\g\to \operatorname{Trep}\g_0$. If we define $\operatorname{Trep}^k\g$ as the subcategory of modules
whose simple submodules are of the form $V(\lambda,\mu)$ with $|\lambda|+|\mu|\leq k$, then the restriction functor maps $\operatorname{Trep}^k\g$ to
$\operatorname{Trep}^k\g_0$.

In a similar way we define the subcategory $(\g_n-\md)^k$ of $\g_n-\md$. It is clear that $\Gamma_n$ maps $\operatorname{Trep}^k\g$ to $(\g_n-\md)^k$.

\begin{lemma}\label{lm:exact} If $n\gg k$, then the functor $\Gamma_n:\operatorname{Trep}^k\g\to (\g_n-\md)^k$ is exact.
\end{lemma}
\begin{proof} Consider the restriction to $\g_0$. It is easy to see that the statement is true for $\operatorname{Trep}^k\g_0$ by semisimplicity of the latter category.
Now the lemma follows from Corollary 
\ref{cor:aboutgamma}.
\end{proof}
\subsection{Injectivity of trivial modules}
\begin{proposition}\label{prop:injectivetrivial} The trivial modules $\mathbb C$ and $\Pi\mathbb C$ are injective in $\operatorname{Trep}\g$.
\end{proposition}
\begin{proof} To prove the statement it is enough to  show that for any strict bipartition $(\lambda,\mu)$  any two exact sequences
$$0\to\mathbb C\to X\to V(\lambda,\mu)\to 0$$
and
$$0\to\Pi\mathbb C\to X\to V(\lambda,\mu)\to 0$$
split. 

First we assume that $V(\lambda,\mu)$ is isomorphic to $\mathbb C$. For the first sequence, we observe that
$\g_1$ acts trivially on $X$ and $\g_0=[\g_1,\g_1]$. Thus, $X$ is a trivial $\g$-module isomorphic to $\mathbb C\oplus\mathbb C$. 
For the second exact sequence, we have a decompositon $X=\mathbb C\oplus\Pi\mathbb C$ of $\g_0$-modules.
By Lemma \ref{lm:invariants} we obtain $\Hom_\g(\mathbb C,X)=\Hom_{\g_0}(\mathbb C,X)=\mathbb C^{1|1}$. Hence
$X$ is isomorphic to $\mathbb C\oplus\Pi\mathbb C$.

Now we assume that $V(\lambda,\mu)$ is not trivial, i.e. that $(\lambda,\mu)$ is a non-empty bipartition. 
Assume that one of the above sequences does not split.
We  use the notations $X(\lambda,\mu)$, $Y_n(\lambda,\mu)$
and $X_n(\lambda,\mu)$ introduced in the proof of Lemma \ref{lm:finitelength}. 
We know that $X$ is a quotient of $X(\lambda,\mu)$. In particular, we have $[X(\lambda,\mu):\mathbb C]\neq 0$.
Recall that $X(\lambda,\mu)=\displaystyle\lim_{\longrightarrow}X_n(\lambda,\mu)$. Hence 
we have $[X_n(\lambda,\mu):\mathbb C]\neq 0$ for  sufficiently large $n$. By Proposition \ref{prop:cohomology}(ii)
this is  possible only if $(\lambda,\mu)=(\square,\square)$. It remains to prove that the sequence splits in this particular case.

We have 
$$X_n(\square,\square)\simeq [\g_n,\g_n]:=\mathfrak{sq}(n).$$
After applying direct limits we obtain
$$X(\square,\square)\simeq [\g,\g]:=\mathfrak{sq}(\infty).$$
But $\mathfrak{sq}(\infty)$ is a simple superalgebra and hence an irreducible $\g$-module, which leads to a contradiction.
\end{proof}

\subsection{Injectivity of $T^{p,q}$}

Sergeev's duality implies that
$Z(\lambda,\mu)$ contains a highest weight vector of weight $(\lambda,\mu)$.
Therefore we know that $V(\lambda,\mu)$ is a subquotient
of $Z(\lambda,\mu)$ and hence of $T^{p,q}$. 

Let $\dot{V}(\lambda,\mu)$ denote the maximal integrable highest weight $\g_0$-module with highest weight $(\lambda,\mu)$. This module is simple (see \cite{DPS}).

\begin{lemma}\label{lm:submodule}  The highest weight $\g$-module $V(\lambda,\mu)$ contains a $\g_0$-submodule isomorphic to $\dot{V}(\lambda,\mu)$. 
\end{lemma}
\begin{proof} Pick up a highest weight vector $v\in V(\lambda,\mu)$ and consider the submodule $U(\g_0)v$. This is simple $\g_0$-module with highest weight
$(\lambda,\mu)$.
\end{proof}

\begin{lemma}\label{lm:hom} 
\begin{itemize}
\item[(i)] If $\Hom_{\g}(V(\lambda,\mu),T^{p,q})\neq 0$, then $|\lambda|=p$ and $|\mu|=q$.
\item[(ii)] $\Hom_{\g}(T^{p,q}, T^{r,s})\neq 0$ implies $p-r = q-s \geq 0$.
\end{itemize} 
\end{lemma}
\begin{proof} (i) Let $\dot{T}^{p,q}=V_0^{\otimes p}\otimes W_0^{\otimes q}$. By Proposition 5.4 in \cite{DPS}, if $\Hom_{\g_0}(\dot{V}(\lambda,\mu),\dot{T}^{p,q})\neq 0$, then $|\lambda|=p$ and $|\mu|=q$. Since $T^{p,q}$ is a direct sum of several copies of $\dot{T}^{p,q}$ the 
statement follows from Lemma \ref{lm:submodule}. Part (ii) follows by similar reasoning.
\end{proof}

%\begin{corollary}\label{cor:zindecomposable} $V(\lambda,\mu)$ is a unique simple submodule of $Z(\lambda,\mu)$. In particular, $Z(\lambda,\mu)$ is an indecomposable $\g$-module.
%\end{corollary}
%\begin{proof}
%Assume that $Z(\lambda,\mu)$ has another simple submodule $V(\lambda',\mu')$. Then $|\lambda'| \leq |\lambda|$ and $ \mu'| \leq  |\mu|$.
%\end{proof}

\begin{lemma}\label{lm:adjoint} Let $M,N,L$ be modules in $\operatorname{Trep}\g$. Then:
$$\Hom(M\otimes L,N)\simeq \Hom(M,\Gamma(\Hom_{\mathbb C}(L,N))).$$
\end{lemma}
\begin{proof} The following isomorphism holds for all $\g$-modules: 
$$\Hom(M\otimes L,N)\simeq \Hom(M,\Hom_{\mathbb C}(L,N)).$$
Now the statement follows directly  from Lemma \ref{lm:gammaadjoint}.
\end{proof}

\begin{lemma}\label{lm:auxilaryinjective} We have the following isomorphisms of $\g$-modules
$$\Gamma(\Hom_{\mathbb C}(V,T^{p,q}))=T^{p,q+1}\oplus (T^{p-1,q})^{\oplus p\theta },\quad \Gamma(\Hom_{\mathbb C}(W,T^{p,q}))=T^{p+1,q}\oplus (T^{p,q-1})^{\oplus q\theta }.$$
\end{lemma}
\begin{proof}
We have $V = V_n \oplus V'$ and $W = W_n \oplus W'$ where $V_n$ (respectively, $W_n$) is the standard (respectively, costandard) $\g_n$-module and  
$V'$ (respectively, $W'$) is the standard (respectively, costandard) $\cc_n$-module. 
Recall that $\Hom_{\cc_n} (V', (V')^{\otimes r} \otimes (W')^{\otimes s} ) \neq 0$ only if $r=1, s=0$ by
Lemma \ref{lm:hom}. Hence we have
\begin{eqnarray*}
\Gamma_n \Hom_{\mathbb C}(V,T^{p,q}) & = & \Hom_{\cc_n} \left( V, V^{\otimes p} \otimes W^{\otimes q} \right)\\  &= & \Hom_{\cc_n} \left( V'\oplus V_n, (V' \oplus V_n)^{\otimes p} \otimes (W' \oplus W_n)^{\otimes q} \right) \\
& \simeq & \Hom_{\cc_n} \left( V', V' \otimes V_n^{\otimes (p-1)} \otimes W_n^{\otimes q} \right)^{\oplus p} \oplus  \Hom_{\cc_n} \left( V_n, V_n^{\otimes p} \otimes W_n^{\otimes q} \right)\\
& \simeq & (V_n^{\otimes (p-1)} \otimes W_n^{\otimes q} )^{\oplus p\theta} \oplus V_n^{\otimes p} \otimes W_n^{\otimes (q+1)}.
\end{eqnarray*}
Then the first identity follows by applying direct limits. We similarly establish the second identity.
\end{proof}

\begin{lemma} \label{lem:endisomorphism}
We have that $\End(T^{p,q}) \simeq {\mathcal H}_{p, q}$.
\end{lemma}

\begin{proof} We have an injective map  ${\mathcal H}_{p, q}  \hookrightarrow \End(T^{p,q})$. 
In order to prove that this is an isomorphism, we compute the dimensions of the two spaces. Using Lemma \ref{lm:auxilaryinjective} we have
\begin{eqnarray*}
 \Hom_{\g}(T^{p,q}, T^{p,q}) & \simeq &   \Hom_{\g}(T^{p-1,q}, \Hom_{\mathbb C}(V,T^{p,q}))\\  & = & \Hom_{\g}(T^{p-1,q}, \Gamma \Hom_{\mathbb C}(V,T^{p,q}))\\
&\simeq&  \Hom_{\g}(T^{p-1,q}, T^{p-1,q})^{\oplus p\theta} \oplus  \Hom_{\g}(T^{p-1,q}, T^{p,q+1})\\
&=&  \Hom_{\g}(T^{p-1,q}, T^{p-1,q})^{\oplus p\theta}.
\end{eqnarray*}
Now by induction on $p+q$ we prove that $\dim  \Hom_{\g}(T^{p,q}, T^{p,q}) = \theta^{p+q} p! q!$ which coincides with the dimension of ${\mathcal H}_{p, q}$.  
\end{proof}

\begin{proposition}\label{prop:injectivetensor} $T^{p,q}$ is injective in $\operatorname{Trep}\g$ for all $p$ and $q$.
\end{proposition}
\begin{proof} In the  case $p = q = 0$ the statement follows from Proposition \ref{prop:injectivetrivial}.
We first assume  that $q>0$. Then using Lemma \ref{lm:auxilaryinjective} we obtain:
\begin{eqnarray*}
 \Hom_{\g}(M \otimes V, T^{p,q-1}) & \simeq & \Hom_{\g} \left( M, \Gamma (\Hom_{\mathbb C}(V,T^{p,q-1} ) \right)\\
&\simeq&  \Hom_{\g} \left( M, T^{p,q}\oplus (T^{p-1,q-1})^{\oplus p\theta} \right)\\
&\simeq&  \Hom_{\g} \left( M, T^{p,q} \right) \oplus \left( \Hom_{\g} \left( M, T^{p-1,q-1} \right) \right)^{\oplus p\theta}.
\end{eqnarray*}
We  apply induction on $q$.   The induction hypothesis implies that the functors  $\Hom_{\g} (\cdot , T^{p-1,q-1} )$ and $\Hom_{\g} (\cdot \otimes V, T^{p,q-1} )$ are exact. Hence,  $\Hom_{\g} (\cdot , T^{p-1,q} )$ is an exact functor. The base case $q=0$ follows by induction on $p$ and by applying the same identitites as above replacing $V$ by $W$. \end{proof}

\begin{proposition}\label{prop:z-inj} $Z(\lambda, \mu)$ is indecomposable injective in $\operatorname{Trep}\g$ with simple socle $V(\lambda, \mu)$.
\end{proposition}
\begin{proof} Let $p = |\lambda|$ and $q = |\mu|$. The injectivity of  $Z(\lambda, \mu)$ follows from  Proposition \ref{prop:injectivetensor} and the fact that   $Z(\lambda, \mu)$ is a direct summand of $T^{p,q}$. The indecomposability of $Z(\lambda, \mu)$ follows from Lemma \ref{lem:endisomorphism} and (\ref{equ:zidempotent}), since $e(\lambda)$ and $e(\mu)$ are primitive idempotents on ${\mathcal H}_p$ and ${\mathcal H}_q$, respectively. 

It remains to show that  the socle of $Z(\lambda, \mu)$ is isomorphic to $V(\lambda, \mu)$. Assume that $V(\lambda', \mu')$ is in the socle of $Z(\lambda, \mu)$ and $(\lambda', \mu') \neq (\lambda, \mu)$. Then looking at the weights of $Z(\lambda, \mu)$ we conclude that $\lambda \geq	 \lambda'$ and $\mu \geq \mu'$ relative to the dominance order of partitions. Moreover $|\lambda'| = p$, $|\mu'| = q$ by Lemma \ref{lm:hom}(i). We now apply induction on $\lambda$ and $\mu$ with respect to the dominance order. For the minimal pair of partitions $\lambda, \mu$ the statement is clear.  By the induction hypothesis on $\lambda', \mu'$, $Z(\lambda', \mu')$ has socle $V(\lambda',\mu')$. Since $V(\lambda',\mu')$ is a submodule of $Z(\lambda, \mu)$, by the injectivity of $Z(\lambda', \mu')$,  we have an injective  homomorphism $Z(\lambda', \mu') \to Z(\lambda,\mu)$. This contradicts with the indecomposability of $Z(\lambda,\mu)$. \end{proof}

\begin{corollary}\label{cor:highestweightsimple} Let $X\in \operatorname{Trep}\g$ be a highest weight module with highest weight $(\lambda,\mu)$. Then
$X$ is isomorphic to $V(\lambda,\mu)$ or $\Pi V(\lambda,\mu)$.
\end{corollary}
\begin{proof} Assume that $V(\lambda', \mu')$ is in the socle of $X$. Then $\lambda \geq	 \lambda'$ and $\mu \geq \mu'$ relative to the dominance order of partitions and we have a nonzero homomorphism $\varphi : X \to Z(\lambda', \mu')$. If $(\lambda, \mu) \neq (\lambda', \mu')$, then a highest weight vector $v$ of $X$ lies in $\ker \varphi$. But $X$ is generated by $v$, therefore $\varphi = 0$ which leads to a contradiction.  Hence $\lambda = \lambda'$, $\mu = \mu'$ and the statement follows. \end{proof}

\begin{corollary}\label{cor:indecomposable} 
 We have 
$$\soc T^{p,q}=\bigoplus_{|\lambda|=p,|\mu|=q,p(\lambda,\mu)=0} V(\lambda,\mu)\boxtimes S(\lambda,\mu)\oplus 
\bigoplus_{|\lambda|=p,|\mu|=q,p(\lambda,\mu)=1} V(\lambda,\mu)\widehat{\boxtimes}S(\lambda,\mu).$$
 \end{corollary}
\begin{proof} The decomposition follows from Proposition \ref{prop:z-inj}  and (\ref{equ:decomp}).
\end{proof}

\begin{corollary}\label{cor:tensor} $\operatorname{Trep}(\g)$ is a symmetric monoidal category (but not rigid!). Furthermore,  the functor 
$$\Gamma_n:\operatorname{Trep}\g\to \g_n-\md$$
is a tensor functor.
\end{corollary}
\begin{proof} We have to check that $\operatorname{Trep}(\g)$ is closed under tensor products. This follows from the injectivity of $T^{p,q}$ and the fact that any module in $\operatorname{Trep}(\g)$ 
is a submodule of a finite direct sum  $\displaystyle\bigoplus_{i=1}^s T^{p_i,q_i}$. Since $\Gamma_n$ is left exact it suffices to check that
$\Gamma_n(M\otimes N)\simeq \Gamma_nM\otimes\Gamma_nN$ for $M=T^{p,q}$ and $N=T^{r,s}$. The latter is straightforward.
\end{proof}

\section{On tensor products and extensions in $\operatorname{Trep}\g$}
\subsection{Diagrammatic description of $\Hom(T^{p,q},T^{r,s})$}
Recall that $\Hom(T^{p,q},T^{r,s})\neq 0$ implies that $p-r=q-s\geq 0$, see Lemma \ref{lm:hom}.

Let $C(p,q,r)=\Hom_\g(T^{p,q}, T^{p-r,q-r})$ and $c(p,q,r)=\dim C(p,q,r)$.
\begin{lemma}\label{lem:formula} For any $p,q,r$ such that $r\leq\min(q,p)$ we have that:
$$c(p,q,r) = \frac{p!q!\theta^{p+q-r}}{r!}$$
\end{lemma}
\begin{proof}
 We will prove the following recursive relation
\begin{eqnarray*}
c(p,q,r)&=&c(p-1,q,r-1)+(p-r)\theta c(p-1,q,r),\\c(p,q,r)&=&c(p,q-1,r-1)+(q-r)\theta c(p,q-1,r).
\end{eqnarray*}
Indeed, we have
\begin{eqnarray*}
 \Hom_{\g}(T^{p,q}, T^{p-r,q-r}) & = &   \Hom_{\g}(T^{p,q-1},\Hom_{\mathbb C}(W,T^{p-r,q-r}))\\  & = & \Hom_{\g} \left( T^{p,q-1}, \Gamma (\Hom_{\mathbb C}(W,T^{p-r,q-r} ) \right)\\
&=&  \Hom_{\g} \left( T^{p,q-1}, T^{p-r+1,q-r}\oplus (T^{p-r,q-r-1})^{\oplus (q-r)\theta} \right).
\end{eqnarray*}
This implies the second recursive relation. The proof of the first one is similar. Now the statement follows easily by induction.
\end{proof}
%\begin{remark}\label{rem:formula}
%In particular we obtain the following useful formula
%$$
%\dim \Hom_{\g}(T^{p,q}, T^{p-r,q-r}) \dim \Hom_{\g}(T^{r,r}, \mathbb C)  = \dim \Hom_{\g}(T^{p,q}, %T^{p,q}).
%$$
%\end{remark}
Our next step is to describe precisely the superspace $C(p,q,r)=\Hom_\g(T^{p,q}, T^{p-r,q-r})$.
For this we will use  diagrams, similar to the ones introduced in \cite{JK}.

Let $D(p,q,r)$ denote the set of diagrams described as follows. Every diagram in  $D(p,q,r)$ has two horizontal rows of nodes with exactly 
$p$ white and $q$ black nodes in the top row, and exactly $p-r$ white and $q-r$ black nodes  in the bottom row.
The nodes are connected by edges that are subject to the following rules.
\begin{itemize}
\item Every node is connected to exactly one node by one edge. In other words we have a prefect pairing.
\item Every node in the bottom row is connected to exactly one node of the same color in the top row.
\item Every node in the top row is connected either to a node of the same color in the bottom row or to a node of the opposite color in the top row.
\item Every edge is either marked or unmarked.
\end{itemize}
If $d\in D(p,q,r)$ and $d'\in D(p-r,q-r,s)$, then we define $d'\cdot d\in D(p,q,r+s)$ by concatenating the diagrams $d$ and $d'$ and removing the middle row. 
An edge of the concatenated diagram 
is marked if the number of marked edges involved the concatenation of the edge is odd. An edge is unmarked if it is not marked. 
An example of a concatenation of three diagrams is presented below.
$$
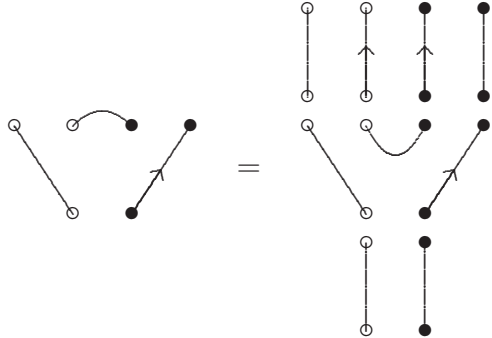
{\beginpicture
\setcoordinatesystem units <0.78cm,0.39cm>
%\put{$\circ$} at  1 0  
\put{$\circ$} at  1 3
\put{$\circ$} at  2 0  \put{$\circ$} at  2 3
\put{$\bullet$} at  3 0   \put{$\bullet$} at  3 3
%\put{$\bullet$} at  4 0  
\put{$\bullet$} at  4 3
\plot 1 3 2 0 /
\plot 4 3 3 0 /
\arrow <3 pt> [1,2] from 3 0 to  3.5 1.5

\setsolid

\put{$=$} at 5 1.5

 \put{$\circ$} at  6 3
\put{$\circ$} at  7 0  \put{$\circ$} at  7 3
\put{$\bullet$} at  8 0  \put{$\bullet$} at  8 3
\put{$\bullet$} at  9 3
\plot 6 3 7 0 /

\plot 9 3 8 0 /
\arrow <3 pt> [1,2] from 8 0 to 8.5 1.5

\put{$\circ$} at  6 4  \put{$\circ$} at  6 7
\put{$\circ$} at  7 4  \put{$\circ$} at  7 7
\put{$\bullet$} at  8 4  \put{$\bullet$} at  8 7
\put{$\bullet$} at  9 4  \put{$\bullet$} at  9 7
\plot 6 4 6 7 /
\plot 9 4 9 7 /
\plot 7 4 7 7 /
\plot 8 4 8 7 /
\arrow <3 pt> [1,2] from 8 4 to 8 5.7
\arrow <3 pt> [1,2] from 7 4 to 7 5.7

\put{$\circ$} at  7 -1  \put{$\circ$} at  7 -4
\put{$\bullet$} at  8 -1  \put{$\bullet$} at  8 -4

\plot 7 -1 7 -4 /
\plot 8 -1 8 -4 /

\setsolid

\setquadratic
\plot 7 3 7.5 2 8 3 /
\plot 2 3 2.5 3.5 3 3 /
%\plot 1 0 2 1 3 0  /
\endpicture}$$

Next we define a map $\gamma:D(p,q,r)\to C(p,q,r)$.  Let  $d\in D(p,q,r)$.  Enumerate the nodes of $d$ in the bottom and in the top row, so that
in the top row the white nodes are labelled by the numbers $1,\dots,p$ and the black nodes are labelled by   $p+1,\dots,p+q$, while in the bottom row, 
the white nodes are labelled by $1,\dots,p-r$ and the  black nodes are labelled by $p+1-r,\dots,p+q-2r$.
Denote by $H^+(d)$ (respectively, $H^-(d)$) the set of pairs $(i,j)$, $i<j$, of nodes in the top row joined by an unmarked (respectively, marked) edge.
For any node $i$ in the bottom row by $s(i)$ we denote the paired to $i$ node in the top row. We let  $m(i)=0$ (respectively, $m(i)=1$) if the   
edge joining $i$ and $s(i)$ is unmarked (respectively, marked). We next introduce the canonical decomposition of $d$ into elementary diagrams $s(p,q,i)$, $o(p,q,i)$, $t(p,q)$ as follows.

The first type of elementary diagrams are 
$s(p,q,i)\in D(p,q,0)$, $i\neq p,p+q$, defined by the conditions $s(j)=j$ if $j\neq i,i+1$, $s(i)=i+1, s(i+1)=i$, and all edges of $s(p,q,i)$ are unmarked.  For example $s(2,0,1)$ is the  diagram:
${\beginpicture
\setcoordinatesystem units <0.78cm,0.39cm>
%\setplotarea x from 0 to 6, y from -2 to 3
%\put{$d_1 = $} at 0 1.5

\put{$\circ$} at  0 -1  \put{$\circ$} at  1 1
\put{$\circ$} at  0 1  \put{$\circ$} at  1 -1
\plot  0 -1 1 1 /
\plot 0 1 1 -1 /
\endpicture}$
\quad while $s(0,2,1)$ is:
${\beginpicture
\setcoordinatesystem units <0.78cm,0.39cm>
%\setplotarea x from 0 to 6, y from -2 to 3
%\put{$d_1 = $} at 0 1.5
\put{$\bullet$} at  0 -1  \put{$\bullet$} at  1 1
\put{$\bullet$} at  0 1  \put{$\bullet$} at  1 -1
\plot  0 -1 1 1 /
\plot 0 1 1 -1 /
%\arrow <3 pt> [1,2] from 1 -1 to  3 1
%\arrow <3 pt> [1,2] from 4 -1 to  1.5 1.5
%\setdashes  <.4mm,1mm>
%\plot 2 -2   2 3 /
\setsolid
\setquadratic
%\put{$i_1$} at 1 -2
%\put{$j_2$} at 4 -2
%\put{$j_1$} at 1 3
%\put{$i_2$} at 4 3
\endpicture}$. We call a \emph{permutation diagram} any diagram formed by the concatenation of diagrams $s(p,q,i)$. The set of all permutation diagrams form a group isomorphic to $S_p\times S_q$. 

Next,  
$o(p,q,i)\in D(p,q,0)$ is the diagram with $s(j)=j$ for all $j=1,\dots p+q$ and with one marked edge joining $i$ with $i$. For example, $o(1,0,1)$ is: $
{\beginpicture
\setcoordinatesystem units <0.6cm,0.3cm>
\setplotarea x from 0.5 to 1.5, y from -1.5 to 2.5
\put{$\circ$} at  1 2
\put{$\circ$} at  1 -1
\plot 1 2  1 -1 /
\arrow <3 pt> [1,2] from 1 -1 to 1 0.75
\endpicture}
$

Finally, let
$t(p,q)\in D(p,q,1)$ be defined by the conditions $H^+(t(p,q))=(p,p+1)$ and $s(i)=i$ for all $i=1,\dots p-1$, $s(i)=i+2$ for $i=p,\dots,p+q-2$. For example, ${\beginpicture
\setcoordinatesystem units <0.78cm,0.39cm>
\setplotarea x from 0 to 4.5, y from -1.5 to 5
\put{$t(1,1)= $} at 0 1.5
\put{$\circ$} at  1.5 2
\put{$\bullet$} at  2.5 2
\setsolid
\setquadratic
\plot 1.5 2 2 2.5 2.5 2  /
\endpicture}$

For any $u_1,\dots,u_p\in V$ and $u_{p+1},\dots, u_{p+q}\in W$  we set $$u=u_1\otimes\dots\otimes u_{p+q}\in T^{p,q}$$ and define 
$$\tilde{s}(p,q,i)(u):=
(-1)^{p(u_i)p(u_{i+1})}u_1\otimes\dots\otimes u_{i-1}\otimes u_{i+1}\otimes u_i\otimes u_{i+2}\otimes\dots\otimes u_{p+q},$$
$$\tilde{o}(p,q,i)(u):=
(-1)^{p(u_1)+\dots+p(u_{i-1})}u_1\otimes\dots\otimes P(u_i)\otimes \dots\otimes u_{p+q},$$
$$\tilde{t}(p,q)(u):=(-1)^{p(u_p)p(u_{p+1})}(u_{p+1},u_p)u_1\otimes \dots \otimes u_{p-1}\otimes u_{p+2}\otimes\dots\otimes u_{p+q}.$$

Note that every $d\in D(p,q,r)$ can be written as a concatenation of elementary diagrams:
$$d=t(p-r+1,q-r+1)\cdot\dots\cdot t(p,q)\cdot o(p,q,i_1)\cdot\dots \cdot o(p,q,i_k)\cdot s(p,q,j_1)\cdot\dots \cdot s(p,q,j_l).$$
For any $d \in D(p,q,r)$, we fix one such decomposition and we set
$$\gamma(d):=\tilde{t}(p-r+1,q-r+1)\circ\dots\circ \tilde{t}(p,q)\circ \tilde{o}(p,q,i_1)\circ\dots \circ \tilde{o}(p,q,i_k)\circ \tilde{s}(p,q,j_1)\circ\dots \circ \tilde{s}(p,q,j_l).$$
Then $\gamma(d)\in C(p,q,r)$ and we have
$$\gamma(d)(u)=
(-1)^{\sigma(u,d)}\prod_{(i,j)\in H^+(d)}(u_j,u_i)\prod_{(i,j)\in H^-(d)}(u_j,Pu_i)P^{m(1)}u_{s(1)}\otimes\dots\otimes P^{m(p+q-2r)}u_{s(p+q-2r)},$$
where the formula for $\sigma (u,d)$ is rather long and is not needed in this paper.
From this formula we see that $\gamma(D(p,q,r))$ is a linearly independent set in $C(p,q,r)$. On the other hand,
Lemma \ref{lem:formula} implies that $c(p,q,r)=|D(p,q,r)|$. Therefore, $\gamma(D(p,q,r))$ forms a basis of $C(p,q,r)$. 
Moreover, from the decomposition of $d$  above we see that $C(p,q,r)$ is generated by
$$u(p,q,r):=t(p-r+1,q-r+1)\circ\dots\circ t(p,q)$$
as a right $\mathcal H(p,q)$-module. The following lemma gives a precise description of $C(p,q,r)$ as an ${\mathcal H}_{p-r,q-r}-{\mathcal H}_{p,q}$-bimodule.

\begin{lemma}\label{lm:homtensor} Consider the embedding $\mathcal H_{r,r}\hookrightarrow \mathcal H_{p,q}$ defined by 
$$s(r,r,1)\mapsto s(p,q,p-r),\dots, s(r,r,r-1)\mapsto s(p,q,p-1),$$
$$s(r,r,r+1)\mapsto s(p,q,p+1),\dots, s(r,r,2r-1)\mapsto s(p,q,p+r-1),$$
$$o(r,r,1)\mapsto o(p,q,p-r),\dots, o(r,r,2r)\mapsto s(p,q,p+r).$$
Then we have the following isomorphism of ${\mathcal H}_{p-r,q-r}-{\mathcal H}_{p,q}$-bimodules:
$$
C(p,q,r) \simeq \Ind_{{\mathcal H}_{r,r}}^{{\mathcal H}_{p,q}} C(r,r,r),
$$
where the definition of the left action of ${\mathcal H}_{p-r,q-r}$ on $ \Ind_{{\mathcal H}_{r,r}}^{{\mathcal H}_{p,q}} C(r,r,r)$ relies on the fact that  
${\mathcal H}_{p-r,q-r}$ and ${\mathcal H}_{r,r}$ are commuting subalgebras of ${\mathcal H}_{p,q}$.
\end{lemma}
\begin{proof} Since $C(p,q,r)$ is generated by $u(p,q,r)$ as a right $\mathcal H_{p,q}$-module and the dimensions of $C(p,q,r)$ and
of $\Ind_{{\mathcal H}_{r,r}}^{{\mathcal H}_{p,q}} C(r,r,r)$ coincide, it remains to verify that the right $\mathcal H_{r,r}$-submodule generated
by $u(p,q,r)$ is isomorphic to $C(r,r,r)$. The latter follows directly from the diagrammatic presentation of $u(p,q,r)$.
\end{proof}
\begin{remark}
The map $\gamma$ is not a homomorphism of diagramatic algebras. However, if $d_1 \in D(p,q,r)$ and $d_2 \in D(p-r,q-r,s)$, then 
$$
\gamma (d_1 \cdot d_2) = (-1)^{\langle d_1, d_2\rangle} \gamma(d_1) \circ \gamma(d_2)
$$
for some function $\langle \cdot , \cdot \rangle : D(p,q,r) \times  D(p-r,q-r,s) \to {\mathbb Z}_2$.
\end{remark}

\subsection{Socle filtrations of $T^{p,q}$ and $Z(\lambda,\mu)$}
\begin{proposition} \label{prop:socle} We have 
$$\soc^r T^{p,q}= \bigcap_{\varphi \in \hom_{\g} (T^{p,q}, T^{p-r,q-r})} \ker \varphi.$$
\end{proposition}
\begin{proof}
It is sufficient to prove the statement for $r=1$ since then we can proceed by induction.
By Corollary \ref{cor:indecomposable} all simple subquotients of $T^{p,q}/\soc T^{p,q}$ are of the form $V(\lambda,\mu)$ or $\Pi V(\lambda,\mu)$
with $|\lambda|<p$ and $|\mu|<q$. Therefore we have an inclusion of $T^{p,q}/\soc T^{p,q}$  into a direct sum of several copies of
$T^{p-r,q-r}$ for different $r$. Hence
$$\soc^1 T^{p,q}= \bigcap_{r \leq \min (p,q)} \bigcap_{\varphi \in \hom_{\g} (T^{p,q}, T^{p-r,q-r})} \ker \varphi.$$
But, using the diagrammatic presentation of $C(p,q,r)$, every $\varphi\in \hom_{\g} (T^{p,q}, T^{p-r,q-r})$ can be factored through
some map $\psi\in\hom_{\g} (T^{p,q}, (T^{p-1,q-1})^{\oplus l})$. Hence $\ker\varphi\subset\ker\psi$ and we obtain
$$\soc^1 T^{p,q}= \bigcap_{\varphi \in \hom_{\g} (T^{p,q}, T^{p-1,q-1})} \ker \varphi.$$ 
\end{proof}

Our next goal is to determine the socle filtration of the indecomposable injective modules $Z(\lambda,\mu)$. For this we need three lemmas.

\begin{lemma}
The following identity  of ${\mathcal H}_{p,q}$-bimodules holds.
$$
{\mathcal H}_{p,q} = \bigoplus_{|\lambda| =p, |\mu| = q} S(\lambda, \mu) \widehat{\boxtimes} S(\lambda, \mu).
$$
\end{lemma}
\begin{proof} The identity follows from Lemma \ref{lem:gensuper}(i).
\end{proof}

\begin{lemma}\label{lem:homz} We have
$$\dim\Hom_{\g} (Z(\lambda, \mu), \mathbb C)=\delta_{\lambda,\mu}.$$
\end{lemma}
\begin{proof} By Lemma \ref{lm:adjoint} we obtain
$$\Hom_{\g} (V(\lambda)\otimes W(\mu), \mathbb C) =\Hom_{\g} (V(\lambda),\Gamma(\Hom_{\mathbb C}(W( \mu), \mathbb C))).$$ 
Lemma \ref{lm:auxilaryinjective} implies
$$\Gamma(\Hom_{\mathbb C}(T^{0,q}, \mathbb C))=T^{q,0}.$$
Since $W(\mu)$ is a direct summand in $T^{0,q}$, it follows
that $\Gamma(\Hom_{\mathbb C}(W( \mu), \mathbb C))=V(\mu)$. Therefore we obtain
$$\dim\Hom_{\g} (Z(\lambda, \mu), \mathbb C) \simeq \begin{cases}\dim\Hom_{\g} (V(\lambda), V(\mu)),\,\,\text{if}\,\, p(\lambda)p(\mu)=0\\
\frac{1}{\theta}\dim\Hom_{\g} (V(\lambda), V(\mu)),\,\,\text{if}\,\, p(\lambda)p(\mu)=1\end{cases},$$
which implies the statement. 
\end{proof}

\begin{lemma} \label{lem:omegas}
The following isomorphism of right ${\mathcal H}_{r,r}$-modules holds.
$$
C(r,r,r)\simeq \bigoplus_{|\gamma| =r} S(\gamma, \gamma).
$$
\end{lemma}
\begin{proof} Substituting $p=q=r$ in (\ref{equ:decomp}) we obtain the decomposition:
\begin{equation}\label{eq:zs}
T^{r,r} = \bigoplus_{|\lambda| =r, |\mu| = r} Z(\lambda, \mu) \widehat{\boxtimes} S(\lambda, \mu).
\end{equation}  

Now, (\ref{eq:zs}) together with Lemma \ref{lem:homz} implies
$$C(r,r,r)=\bigoplus_{|\gamma|=r}\Hom_{\g}(Z(\gamma, \gamma),\mathbb C) \otimes S(\gamma, \gamma)=\bigoplus_{|\gamma| =r} S(\gamma, \gamma).$$
\end{proof}

%The following is easy to verify.

%\begin{lemma} \label{lem:productpower}
%If $M_1$ is an ${\mathcal H}_p$-module and $M_2$ is an ${\mathcal H}_q$-module, then 
%$$
%[M_1 \boxtimes M_2 : S(\lambda, \mu)] = \theta^{p(\lambda)p(\mu)}[M_1:S(\lambda)][M_2: S(\mu)]
%$$
%whenever $|\lambda| = p$ and $|\mu| = q$.
%\end{lemma}

\begin{theorem}\label{thm:zhom}
The following identity holds for $|\lambda| - |\lambda'| = |\mu| - |\mu'| = r$:
$$
\dim \Hom_{\g} (Z(\lambda, \mu), Z(\lambda', \mu')) = 
\frac{1}{\theta^{p(\lambda)p(\mu)}\theta^{p(\lambda')p(\mu')}}\sum_{|\gamma| = r} \frac{1}{\theta^{p(\gamma)}}f_{\lambda', \gamma}^{\lambda} f_{\mu', \gamma}^{\mu}. 
$$
\end{theorem}

\begin{proof} Let $|\lambda|=p$ and $|\mu| = q$.
Using Lemma \ref{lm:homtensor} and (\ref{equ:zidempotent}) we obtain
$$\dim\Hom_{\g} (Z(\lambda, \mu), Z(\lambda', \mu'))=
\frac{\dim e(\lambda')\otimes e( \mu')C(p,q,r)  e(\lambda)\otimes e(\mu)}{\theta^{p(\lambda)p(\mu)}\theta^{p(\lambda')p(\mu')}} .$$

Recall that for any right $\mathcal H_{p,q}$-module $M$, 
$$\dim M e(\lambda)\otimes e(\mu)=\dim\Hom_{\mathcal H_{p,q}}(M, S(\lambda)\boxtimes S(\mu)).$$
Next, Lemma \ref{lm:homtensor} implies
\begin{eqnarray*}&&e(\lambda')\otimes e( \mu')C(p,q,r)=
e(\lambda')\otimes e(\mu')\left( \Ind_{{\mathcal H}_{r,r}}^{{\mathcal H}_{p,q}} C(r,r,r) \right)\\
&=& e(\lambda')\otimes e(\mu')\left( \Ind_{{\mathcal H}_{r,r} \otimes {\mathcal H}_{p-r, q-r}}^{{\mathcal H}_{p,q}} C(r,r,r) \boxtimes  {\mathcal H}_{p-r, q-r}  \right)\\ 
&=&\Ind_{{\mathcal H}_{r,r} \otimes  {\mathcal H}_{p-r, q-r}}^{{\mathcal H}_{p,q}} C(r,r,r) \boxtimes  e(\lambda')\otimes e(\mu') {\mathcal H}_{p-r, q-r}   \\
&=&\Ind_{{\mathcal H}_{r,r} \otimes  {\mathcal H}_{p-r, q-r}}^{{\mathcal H}_{p,q}} C(r,r,r) \boxtimes  S(\lambda')\boxtimes S(\mu').
\end{eqnarray*}

Finally, using Lemma \ref{lem:omegas} and Lemma \ref{lem:LRind} we obtain
\begin{eqnarray*}  && \dim\Hom_{\mathcal H_{p,q}}\left(\Ind_{{\mathcal H}_{r,r} \otimes  {\mathcal H}_{p-r, q-r}}^{{\mathcal H}_{p,q}} C(r,r,r) \boxtimes  S(\lambda')\boxtimes S( \mu'), 
S(\lambda)\boxtimes S(\mu)\right) \\ &=&  \sum_{|\gamma| = r}\frac{1}{\theta^{p(\gamma)}}\dim\Hom_{\mathcal H_{p,q}}\left
(\Ind_{{\mathcal H}_{r} \otimes  {\mathcal H}_{p-r}}^{{\mathcal H}_{p}}  (S(\gamma) \boxtimes S(\lambda')) \boxtimes  \Ind_{{\mathcal H}_{r} \otimes  {\mathcal H}_{q-r}}^{{\mathcal H}_{q}} (S(\gamma) \boxtimes S(\mu')),S(\lambda)\boxtimes S(\mu)\right) \\ 
&=&  
\sum_{|\gamma| = r}\frac{1}{\theta^{p(\gamma)}}\dim\Hom_{\mathcal H_{p}}\left(\Ind_{{\mathcal H}_{r} \otimes  {\mathcal H}_{p-r}}^{{\mathcal H}_{p}}  
S(\gamma) \boxtimes S(\lambda'),S(\lambda)\right)\\&&
\dim\Hom_{\mathcal H_{q}}\left( \Ind_{{\mathcal H}_{r} \otimes  {\mathcal H}_{q-r}}^{{\mathcal H}_{q}} (S(\gamma) \boxtimes S(\mu')),S(\mu)\right)\\ 
& = &\sum_{|\gamma| = r}\frac{1}{\theta^{p(\gamma)}} f_{\lambda', \gamma}^{\lambda} f_{\mu', \gamma}^{\mu},
\end{eqnarray*}
which completes the proof.
\end{proof}

We set $\soc_r M = \soc^{r+1}M/soc^r M$.

\begin{corollary}\label{cor:socfilt}
$$
[\soc_r Z(\lambda, \mu) : V(\lambda', \mu')] =
\frac{1}{\theta^{p(\lambda)p(\mu)}\theta^{p(\lambda')p(\mu')}\theta^{p(\lambda,\mu)p(\lambda',\mu')+p(\lambda,\mu)+p(\lambda',\mu')}}\sum_{|\gamma| = r} \frac{1}{\theta^{p(\gamma)}}f_{\lambda', \gamma}^{\lambda} f_{\mu', \gamma}^{\mu}.
$$
\end{corollary}
\begin{proof} The identity follows from Theorem \ref{thm:zhom} and the relation
$$[\soc_r Z(\lambda, \mu) : V(\lambda', \mu')] =\frac{\dim\Hom(Z(\lambda,\mu),Z(\lambda',\mu'))}{\theta^{p(\lambda,\mu)p(\lambda',\mu')+p(\lambda,\mu)+p(\lambda',\mu')}}.$$
\end{proof}

\subsection{Extensions and blocks}
\begin{corollary}\label{cor:extension} If $\ext^1(V(\lambda',\mu'),V(\lambda,\mu))\neq 0$, then  $\lambda\in\lambda'+\square$ and $\mu\in\mu'+\square$.
Furthermore we have the following cases:
\begin{enumerate} 
\item If both $V(\lambda',\mu')$ and $V(\lambda,\mu)$ are of M-type, then
$$\mathbb C=\begin{cases}\ext^1(V(\lambda',\mu'),V(\lambda,\mu))=\ext^1(V(\lambda',\mu'),V(\lambda,\mu)),\,\text{if}\,\,p(\lambda)=p(\lambda'),p(\mu)=p(\mu')\\
\ext^1(V(\lambda',\mu'),V(\lambda,\mu))\oplus\ext^1(V(\lambda',\mu'),\Pi V(\lambda,\mu)),\,\,\text{otherwise},
\end{cases}$$
\item If $V(\lambda',\mu')$ is of Q-type and $V(\lambda,\mu)$ is of M-type, then
$$\ext^1(V(\lambda',\mu'),V(\lambda,\mu))\oplus\ext^1(V(\lambda',\mu'),\Pi V(\lambda,\mu))=\mathbb C,$$
\item If $V(\lambda',\mu')$ is of M-type and $V(\lambda,\mu)$ is of Q-type, then
$$\ext^1(V(\lambda',\mu'),V(\lambda,\mu))\oplus\ext^1(\Pi V(\lambda',\mu'),V(\lambda,\mu))=\mathbb C,$$
\item If both $V(\lambda',\mu')$ and $V(\lambda,\mu)$ are of Q-type, then
$$\ext^1(V(\lambda',\mu'),V(\lambda,\mu))=\begin{cases}\mathbb C^2,\,\,\text{if}\,\,p(\lambda)=p(\lambda'),p(\mu)=p(\mu')\\
\mathbb C,\,\,\text{otherwise}.
\end{cases}$$
\end{enumerate}
\end{corollary}
\begin{proof} Straightforward calculation using Corollary \ref{cor:socfilt} and (\ref{equ:tensorproduct}).
\end{proof}

The following is an immediate consequence of Corollary \ref{cor:extension}.
\begin{corollary}\label{cor:blocks} Let $\operatorname{Trep}_m\g$ be the full subcategory of $\operatorname{Trep}\g$ with simple objects
$V(\lambda,\mu),\Pi V(\lambda,\mu)$ for all $(\lambda,\mu)$ such that $|\lambda|-|\mu|=m$. Then we have the decomposition
$$\operatorname{Trep}\g=\bigoplus_{m\in\mathbb Z}\operatorname{Trep}_m\g.$$
\end{corollary}
\begin{proposition}\label{prop:blocks} For any $m\in\mathbb Z$ the subcategory $\operatorname{Trep}_m\g$ is an indecomposable block.
\end{proposition}
\begin{proof} We define an equivalence relation on isomorphism classes of simple modules of $\operatorname{Trep}_m\g$. We say 
$X\prec  Y$ if $\ext^1(X,Y)\neq 0$, and set $\sim$ be the minimal equivalence relation containing $\prec$.
We have to prove that isomorphism classes of simple modules of $\operatorname{Trep}_m\g$ form one equivalence class.
Note that 
\begin{equation}\label{eq:equiv}
X\sim Y\,\,\Rightarrow\,\,\Pi X\sim \Pi Y.
\end{equation}

Using symmetry we can assume without loss of generality that $m\geq 0$. We first claim that $V(\lambda,\mu)$ is equivalent to
$V(\eta,\emptyset)$ or $\Pi V(\eta,\emptyset)$ for some partition $\eta$ with $|\eta|=m$. Indeed, take $\lambda'\in\lambda-\square$
and $\mu'\in\mu-\square$, then we have $V(\lambda',\mu')\prec V(\lambda,\mu)$. Thus, we can decrease $|\lambda|$ and $|\mu|$ by $1$
and proceed by induction. 

Next we show that  $V(\eta,\emptyset)\sim\Pi V(\eta,\emptyset)$. Indeed, the statement is non-trivial only if $V(\eta,\emptyset)$ is of M-type,
If $m>0$ consider $\eta'$ obtained from $\eta$ by adding $\square$ in the first row. Then $V(\eta',\square)$ is of Q-type and
we have $$V(\eta,\emptyset)\prec V(\eta',\square),\,\,\Pi V(\eta,\emptyset)\prec V(\eta',\square).$$
If $m=0$ we have to show $\Pi \mathbb C\sim\mathbb C$. For this set 
$$\lambda=\yng(3)\,,\quad \mu=\yng(2,1).$$
Then $V(\lambda,\mu)$ is of Q-type and equivalent to both $\mathbb C$ and $\Pi\mathbb C$.

If we start with the partition $\eta$
having one row with $m$ boxes, we can obtain from it
any other strict partition of size $m$ in several steps, where each step consists of moving a box from the top row to some other row.
If $\eta''$ is obtained from $\eta'$ in one step, consider the partition $\nu$ obtained from $\eta''$ be adding a box in the first row.
Then we have $V(\eta'',\emptyset)\sim V(\nu,\square)\sim V(\eta',\emptyset)$. Therefore $V(\kappa,\emptyset)\sim V(\eta,\emptyset)$ for all $\kappa$ of size $m$.
The proof is complete.
\end{proof}

\begin{lemma}\label{lem:resolution} Any $M\in\operatorname{Trep}\g$ has a finite injective resolution. If
$M=V(\lambda,\mu)$ and
$$0\to R^0\to R^1\to\dots \to R^k\to 0$$ 
is the minimal injective resolution of $M$, then $[R^i:Z(\lambda',\mu')]\neq 0$ implies $|\lambda|-|\lambda'|=|\mu|-|\mu'|\geq i$.
In particular, $k\leq \min(|\lambda|,|\mu|)+1$.
\end{lemma}
\begin{proof} Since $\operatorname{Trep}\g$ has enough injectives we only need to check finiteness of the minimal injective resolution.
Let  $V(\lambda,\mu)$ be a simple submodule of $\soc M$ with maximal $|\lambda|+|\mu|=s$. Consider an embedding $\varphi: M\hookrightarrow R^0$, 
where $R^0$ is the injective hull of $\soc M$, then by Corollary \ref{cor:socfilt} all simple subquotients $V(\lambda',\mu')$ in $\operatorname{coker}\varphi$
satisfy $|\lambda'|+|\mu'|<s$. That shows that the length of resolution is at most $s+1$ and in the case $M=V(\lambda,\mu)$ implies the last assertion.
\end{proof}

\begin{corollary}
If $\ext^i \left( V(\lambda', \mu'), V(\lambda, \mu)\right) \neq 0$ then $|\lambda|-|\lambda'|=|\mu|-|\mu'|\geq i$.
\end{corollary}

\subsection{Tensor products} In this subsection we find formulas for the tensor products of the indecomposable injectives in $\operatorname{Trep}\g$. The formulas are relatively easy to obtain.
\begin{lemma} \label{lem:zcoeff} We have
$$Z(\lambda, \mu) \otimes Z(\lambda', \mu')=\bigoplus_{|\lambda"|=|\lambda|+|\lambda'|,|\mu''|=|\mu|+|\mu'|}Z(\lambda'', \mu'')^{\oplus s(\lambda'',\mu'')},$$
where
$$s(\lambda'',\mu'')=\frac{\theta^{p(\lambda'')p(\mu'')}f^{\lambda''}_{\lambda,\lambda'}f^{\mu''}_{\mu,\mu'}}{\theta^{p(\lambda)p(\mu)}\theta^{p(\lambda')p(\mu')}
\theta^{p(\lambda'')}\theta^{p(\mu'')}} $$
\end{lemma}
\begin{proof} The identity follows by direct computation using the definitions of $Z(\lambda, \mu)$ and $f^{\nu}_{\lambda,\mu}$. 
\end{proof}

\begin{corollary}\label{cor:zcoeff} We have
$$Z(\lambda,\mu)\otimes V=\bigoplus_{\lambda'\in\lambda+\square}Z(\lambda',\mu)^{\oplus u(\lambda',\lambda,\mu)},\; Z(\lambda,\mu)\otimes W=\bigoplus_{\mu'\in\mu+\square}Z(\lambda,\mu')^{\oplus u(\mu',\mu,\lambda)},$$
where 
$$u(\alpha',\alpha,\beta)=\begin{cases}1,\,\,\text{if}\,\, p(\alpha,\beta)=0,\,p(\alpha',\beta)=1,\\
\theta,\,\text{otherwise}.\end{cases}$$
\end{corollary}

\begin{proposition}\label{prop:translation} The tensor products $V(\lambda,\mu)\otimes V$ and $V(\lambda,\mu)\otimes W$ have Loewy length at most 2.
Furthermore,
$$\soc(V(\lambda,\mu)\otimes V)=\bigoplus_{\lambda'\in\lambda+\square}V(\lambda',\mu)^{\oplus u(\lambda',\lambda,\mu)},$$
$$\soc(V(\lambda,\mu)\otimes W)=\bigoplus_{\mu'\in\mu+\square}V(\lambda,\mu')^{\oplus u(\mu',\mu,\lambda)},$$
and
$$\soc_2(V(\lambda,\mu)\otimes V)=\bigoplus_{\mu'\in\mu-\square}V(\lambda,\mu')^{\oplus u(\mu,\mu',\lambda)},$$
$$\soc_2(V(\lambda,\mu)\otimes W)=\bigoplus_{\lambda'\in\lambda-\square}V(\lambda',\mu)^{\oplus u(\lambda,\lambda',\mu)},$$
where $u(\alpha',\alpha,\beta)$ is defined in Corollary \ref{cor:zcoeff}.
\end{proposition}
\begin{proof} Let $|\lambda|=p, |\mu|=q$. Recall that $V(\lambda,\mu)$ is a submodule of $\soc T^{p,q}$. Hence
$V(\lambda,\mu)\otimes V$ is a submodule of $T^{p+1,q}$. We now use Proposition \ref{prop:socle}. Note that 
$$V(\lambda,\mu)\otimes V\subset\ker\varphi$$ for any $\varphi\in C(p+1,q,2)$ since any such $\varphi$ involves two contractions.
Hence the Loewy length of $V(\lambda,\mu)\otimes V$ is at most $2$.

To obtain $\soc (V(\lambda,\mu)\otimes V)$ we use that
$$\soc (V(\lambda,\mu)\otimes V)=\soc (Z(\lambda,\mu)\otimes V)$$
and Corollary \ref{cor:zcoeff}.

To compute $\soc_2(V(\lambda,\mu)\otimes V)$ we first note that
$$[\soc_2(V(\lambda,\mu)\otimes V):V(\lambda'',\mu'')]\neq 0\,\, \Rightarrow \,\,|\lambda''|=|\lambda|, |\mu''|=|\mu|-1.$$
Furthermore,
$$\hom(V(\lambda,\mu)\otimes V, Z(\lambda'',\mu''))=\hom (V(\lambda,\mu),\Gamma(\Hom_{\mathbb C}(V,Z(\lambda'',\mu''))))$$
and
$$\Gamma(\Hom_{\mathbb C}(V,Z(\lambda'',\mu'')))=Z(\lambda'',\mu'')\otimes W\oplus S$$
for some $S\subset (T^{p-1,q-1})^{\oplus p\theta}$. Taking into account that $\hom (V(\lambda,\mu),S)=0$, we obtain
$$\hom(V(\lambda,\mu)\otimes V, Z(\lambda'',\mu''))=\hom(V(\lambda,\mu), Z(\lambda'',\mu'')\otimes W).$$
By Corollary \ref{cor:zcoeff} we know the decomposition of $Z(\lambda'',\mu'')\otimes W$. As a result, we see that 
$$\hom(V(\lambda,\mu), Z(\lambda'',\mu'')\otimes W)\neq 0\,\,\Rightarrow\,\,\lambda''=\lambda,\mu\in\mu'+\square.$$ Moroever,
$$[\soc_2(V(\lambda,\mu)\otimes V):V(\lambda'',\mu'')]=\dim \hom(V(\lambda,\mu), Z(\lambda'',\mu'')\otimes W)=u(\mu,\mu',\lambda).$$
This completes the proof for the identities involving $V(\lambda,\mu)\otimes V$. The identities involving $V(\lambda,\mu)\otimes W$ follow by similar reasoning.
\end{proof}

\section{ Koszulity of $\operatorname{Trep}\g$}
\begin{theorem}\label{th:koszulity} The category $\operatorname{Trep}\g$ is Koszul.
\end{theorem} 
\begin{proof} For any bipartition $(\lambda,\mu)$  we set
$$d(\lambda,\mu):= \min ( |\lambda|, |\mu|).$$
Let 
$$0\to R^0(\lambda,\mu)\to  R^1(\lambda,\mu)\to\dots\to R^k(\lambda,\mu)\to 0$$
be the minimal injective resolution of $V(\lambda,\mu)$ (note that the resolution is finite by Lemma \ref{lem:resolution}).
The Koszulity of $\operatorname{Trep}\g$ is equivalent to each of the following two equivalent statements:
\begin{enumerate}
\item $[R^i(\lambda,\mu):Z(\lambda',\mu')]\neq 0$ implies $d(\lambda,\mu)=d(\lambda',\mu')+i$;
\item $\ext^i(V(\lambda',\mu'),V(\lambda,\mu))\neq 0$ implies $d(\lambda,\mu)=d(\lambda',\mu')+i$.
\end{enumerate}
Indeed, (1) is equivalent to Koszulity since $d(\cdot, \cdot)$ induces the grading on $\operatorname{Trep}\g$. Furthermore, (1) obviously implies (2). To show that
(2) implies (1) assume the opposite, i.e. that there exists $(\lambda',\mu')$ such that $d(\lambda,\mu)=d(\lambda',\mu')+i$ and 
$[R^j(\lambda,\mu):Z(\lambda',\mu')]\neq 0$ for some $j\neq i$.
Lemma \ref{lem:resolution} implies $j<i$. Let us choose the minimal such $j$.
Since $\ext^j(V(\lambda,\mu),V(\lambda',\mu'))=0$, the map $Z(\lambda',\mu')\to R^{j+1}(\lambda,\mu)$ must be injective, which contradicts the minimality of
the resolution.

Without loss of generality we assume that $|\lambda| \leq |\mu|$, i.e. $d (\lambda, \mu) = |\lambda|$. We prove (2) for all $\lambda, \mu$ by induction on $|\lambda|$. The base case $\lambda=\emptyset$ follows from the fact that $V(\emptyset,\mu)$ is injective.
To prove the inductive step pick up $\nu\in\lambda-\square$.  Recall that $V(\nu,\mu) \otimes V$ has Loewy length $2$ by Proposition \ref{prop:translation}. Consider the exact sequence
\begin{equation}\label{eq:soc}
0\to \soc(V(\nu,\mu)\otimes V)\to V(\nu,\mu)\otimes V\to\soc_2(V(\nu,\mu)\otimes V)\to 0
\end{equation}
and the minimal resolution
$$0\to R^0(\nu,\mu)\to  R^1(\nu,\mu)\to\dots\to R^k(\nu,\mu)\to 0$$
of $V(\nu,\mu)$. Note that by Proposition \ref{prop:translation}, all simple components of $\soc_2(V(\nu,\mu)\otimes V)$ satisfy the induction hypothesis. Therefore, 
\begin{equation}\label{eq:cond1}\ext^j(V(\lambda',\mu'),\soc_2(V(\nu,\mu)\otimes V))\neq 0\,\,\Rightarrow\,\, i = |\nu| - |\lambda'| = |\lambda| - |\lambda'| -1.
\end{equation}

This resolution satisfies (1) by the induction hypothesis. 
We have that  
$$0\to R^0(\nu,\mu)\otimes V\to  R^1(\nu,\mu)\otimes V\to\dots\to R^k(\nu,\mu)\otimes V\to 0.$$
is an injective resolution of $V(\nu,\mu)\otimes V$. Since $[R^i(\nu,\mu):Z(\nu',\mu')]\neq 0$ implies $i = |\nu| - |\nu'|$, by Corollary \ref{cor:zcoeff}  we have  that 
\begin{equation}\label{eq:cond2}[R^i(\nu,\mu) \otimes V:Z(\lambda',\mu')]\neq 0\,\,\Rightarrow\,\, i = |\nu| - |\lambda'| +1 = |\lambda| - |\lambda'|.
\end{equation}
Equivalently, 
$$\ext^j( V(\lambda',\mu'),V(\nu,\mu)\otimes V))\neq 0\,\,\Rightarrow\,\,i = |\lambda| - |\lambda'|.$$

Therefore, by (\ref{eq:cond1}) and (\ref{eq:cond2}) the long exact sequence $\ext^{\cdot}(V(\lambda', \mu'), \cdot)$ associated to (\ref{eq:soc}) gives $\ext^{i}(V(\lambda', \mu'), \soc(V(\nu, \mu) \otimes V)) = 0$ for $i \neq |\lambda| - |\lambda'|$. Since $V(\lambda,\mu)$ is a direct summand in $\soc(V(\nu,\mu)\otimes V)$, we prove that  condition (2) holds for $V(\lambda,\mu)$.
\end{proof}

Recall that  $\mathbb T =\bigoplus T^{p,q}$. Set $\mathbb T_{> k}=\bigoplus_{p+q > k}  T^{p,q}$ and  $\mathbb T_{\leq k}=\bigoplus_{p+q \leq k}  T^{p,q}$. Let also
$$
A_{(k)}=\left\{ \varphi \in \End_{\g} \mathbb T\; | \; \varphi(\mathbb T_{> k}) = 0\right\}.
$$
Clearly, $A_{(k)} \simeq \End \left( \mathbb T_{\leq k}\right) $.  By $A_{(k)}$-mod we denote the category of finite-dimensional ${\mathbb Z}_2$-graded $A_{(k)}$-modules.

We  have a chain of monomorphisms
$$
A_{(1)} \subset A_{(2)} \subset \dots
$$
Note that the unit of $A_{(k)}$ does not map to the unit of $A_{(k+1)}$ under the embedding $ A_{(k)}\hookrightarrow A_{(k+1)}$. We set 
$$
A = \lim_{\longrightarrow} A_{(k)}, \; A{\rm - mod} = \lim_{\longrightarrow} \left( A_{(k)}{\rm - mod}\right)$$
Note that $A$ is not unital and that the category $A{\rm - mod}$ consists of all  finite-dimensional ${\mathbb Z}_2$-graded $A$-modules $X$ such that $AX=X$.

\begin{theorem} The functors $\Hom_\g(\cdot,\mathbb T)$ and $\Hom_A(\cdot,\mathbb T)$ establish an antiequivalence of the categories $\operatorname{Trep}\g$ and $A$-mod.
\end{theorem}  
\begin{proof}
Recall that $\displaystyle \operatorname{Trep}\g = \lim_{\longrightarrow} \operatorname{Trep}^k\g $. By Proposition \ref{prop:injectivetensor} and Corollary \ref{cor:indecomposable}, $\mathbb T_{\leq k}$ is  an injective cogenerator of $\operatorname{Trep}\g$. In order to prove the statement, it is sufficient to show that the functors $\Phi:= \Hom_{\g} (\cdot, \mathbb T_{\leq k})$ and $\Psi:= \Hom_{A_{(k)}} (\cdot, \mathbb T_{\leq k})$ establish an antiequivalence of the categories $ \operatorname{Trep}^k\g $ and $A_{(k)}{\rm - mod}$.  We have that $\Phi$ is an exact functor since $\mathbb T_{\leq k}$ is an injective module in $ \operatorname{Trep}^k\g$. Therefore, $\Psi \Phi$ is a left exact functor, and $\Phi \Psi$ is a right exact functor.

We first note that for all $X \in  A_{(k)}{\rm - mod}$ and $M \in \operatorname{Trep}^k\g $ we have isomorphisms 
$$
\Hom_{A} (X, \Phi M) \simeq \Hom_{A \times \g} (X \otimes M, \mathbb T_{\leq k} )  \simeq \Hom_{\g} (M, \Psi X).
$$
Using the isomorphisms 
$$
\Hom_{\g} (\Psi X, \Psi X) \simeq \Hom_{A} (X, \Phi \Psi X), \; \Hom_{A} (\Phi M, \Phi M) \simeq \Hom_{\g} (M, \Psi \Phi M), 
$$
we define morphisms $\alpha_X: X \to \Phi \Psi X$ and $\beta_M: M \to \Psi \Phi M$. To complete the proof, it is sufficient to verify that $\alpha_X$ and $\beta_M$ are isomorphisms for all $X \in  A_{(k)}{\rm - mod}$ and all $M \in \operatorname{Trep}^k\g$. Note that this is true for simple modules by Corollary \ref{cor:indecomposable}. 

We first prove that $\beta_M$ is an isomorphism using induction on the length of $M$. As mentioned above, $\beta_M$ is isomorphism for simple modules $M$ which implies the base case. Let 
\begin{equation}
\begin{CD}
0@>>> N @>>>M @>\sigma>>L @>>>0
\end{CD}\notag
\end{equation}
be a short exact sequence of modules in $\operatorname{Trep}^k\g $. Consider the induced diagram

\begin{equation}
\begin{CD}
0@>>> N @>>>M @>\sigma>>L @>>>0\\
@VVV @V\beta_NVV @V\beta_MVV @V\beta_LVV\\
0@>>>\Psi \Phi N @>>>\Psi \Phi M @>\Psi \Phi (\sigma)>> \Psi \Phi L
\end{CD}\notag
\end{equation}
\medskip

By the induction hypothesis, $\beta_N$ and $\beta_L$ are isomorphisms. Therefore  $\beta_L \sigma$ is surjective which implies that $\Psi \Phi (\sigma)$ is surjective. By the Five Lemma, $\beta_M$ is an isomorphism.

We last show that $\alpha_{X}$ is an isomorphism. Note that $\alpha_{A_{(k)}}$ is an isomorphism and hence  $\alpha_{Z}$ is an isomorphism for any free $A_{(k)}$-module $Z$ of finite rank. Any $X$ in $A_{(k)}$-mod can be included in a short exact sequence 
\begin{equation}
\begin{CD}
0@>>> Y @>\tau>>Z @>\varphi>>X @>>>0
\end{CD}\notag
\end{equation}
for some free $A_{(k)}$-module $Z$ of finite rank. Consider  the induced diagram

\begin{equation}
\begin{CD}
0@>>> Y @>\tau>>Z @>\varphi>>X @>>>0\\
@. @V\alpha_YVV @V\alpha_ZVV @V\alpha_XVV @VVV\\\
@.\Phi \Psi Y @>\Phi \Psi (\tau)>>\Phi \Psi Z @>\Psi \Phi (\varphi)>> \Phi \Psi X @>>>0
\end{CD}\notag
\end{equation}
Since $\alpha_Z$ is an isomorphism,  $\alpha_X$ is surjective for any module $X$. In particular, $\alpha_Y$ is surjective. On the other hand, $\Phi \Psi (\tau) \alpha_Y = \alpha_Z \tau$ is injectve and thus $\alpha_Y$ and $\Phi \Psi (\tau)$ are injective as well. By the Five Lemma, $\alpha_X$ is an isomorphism. \end{proof}

\begin{proposition}
$A_{(k)}$ is a Koszul self-dual superalgebra.
\end{proposition}
\begin{proof} We follow the notation and definitions of Section 2 of \cite{BGS}. The Koszulity of $A_{{k}}$ follows from the Koszulity of $\operatorname{Trep}^k\g$. Then
$$
A_{(k)} = \bigoplus_{r\geq 0}A_{(k)}^r 
$$
where  $A_{(k)}^r =\bigoplus_{p+q \leq r} C(p,q,r)$. In particular, $A_{(k)}^0 =\bigoplus_{p+q \leq r} {\mathcal H}_{p,q}$ is a semisimple superalgebra. From Lemma \ref{lm:homtensor}, 
$$
C(1,1,1) \otimes_{\mathcal H_{1,1}} \mathcal H_{p,q} \simeq C(p,q,1).
$$
Furthermore,
$$
C(p-1,q-1,1) \otimes_{\mathcal H_{p-1,q-1}} C(p,q,1) \simeq  \left( C(1,1,1) \boxtimes C(1,1,1) \right) \otimes_{\mathcal H_{1,1} \otimes \mathcal H_{1,1} } \mathcal H_{p,q}.
$$
Therefore,
$$
A_{(k)}^1 \otimes_{A_{(k)}^0} A_{(k)}^1 \simeq  \bigoplus_{p+q \leq k} \left( C(1,1,1) \boxtimes C(1,1,1) \right) \otimes_{\mathcal H_{1,1} \otimes \mathcal H_{1,1} } \mathcal H_{p,q}.
$$
and the quadratic relations submodule of $A_{(k)}^1 \otimes_{A_{(k)}^0} A_{(k)}^1 $  is generated by the elements $x \otimes y - (-1)^{p(x)p(y)}y \otimes x $, $x,y \in C(1,1,1)$.
Let $B_{(k)} = \left( A_{(k)}^{!}\right)^{\rm opp}$ be the Koszul dual of $A_{(k)}$. Then $A_{(0)} = B_{(0)}$, $A_{(1)} = B_{(1)}$, and 
$$
B_{(k)}^1 \otimes_{B_{(k)}^0} B_{(k)}^1 \simeq A_{(k)}^1 \otimes_{A_{(k)}^0} A_{(k)}^1.
$$ 
The quadratic relations submodule of $B_{(k)}^1 \otimes_{B_{(k)}^0} B_{(k)}^1 $  is generated by the elements $x \otimes y + (-1)^{p(x)p(y)}y \otimes x $, $x,y \in C(1,1,1)$.

Let $U = A_{(k)}^1 = B_{(k)}^1$. Then $A_{(k)} = T(U)/(R)$ and $B_{(k)} = T(U)/(R^{\perp})$. Consider the automorphism $\gamma$ of $A_{(k)}^0$ defined by $s(p,q,i) \mapsto s(p,q,i)$ if $i>p$, $s(p,q,i) \mapsto - s(p,q,i)$ if $i <p$, $o(p,q,j) \mapsto o(p,q,j)$. Then $U^{\gamma} = U$ and $\gamma$ extends to an automorphism $\widetilde{\gamma} : T(U) \to T(U)$ such that $\widetilde{\gamma}(R) = R^{\perp}$. Hence $A_{(k)}$ is isomorphic to $B_{(k)}$.
\end{proof}

\begin{corollary} We have that
$$\dim\operatorname{ext}^i(V(\lambda',\mu'),V(\lambda,\mu))=[\soc^{i+1}Z(\lambda,\mu):V(\lambda',\mu')],$$
and the latter are computed in Corollary \ref{cor:socfilt}.
\end{corollary}

\end{document}